\documentclass[11pt,a4paper]{article}

\usepackage[utf8]{inputenc}
\usepackage[T1]{fontenc}
\usepackage{microtype}      
\usepackage{lmodern}             
\usepackage{amsmath,amssymb,amsthm,mathtools}
\usepackage{mathrsfs}            
\usepackage{enumitem}            
\usepackage{graphicx}             
\usepackage{hyperref}             
\usepackage{geometry}
\usepackage{bbm}                 
\usepackage{cleveref}             
\usepackage{xspace}              
\usepackage{color}
\usepackage{faktor}
\usepackage{caption}
\DeclareMathOperator{\sech}{sech}

\hypersetup{
  pdftitle={Ruling out Infinite-Time Bubbling under SO(4)-Equivariant Yang--Mills Flow on S4 x SU(2)},
  pdfauthor={Gabriel Tribe},
  colorlinks=true,
  linkcolor=blue,
  citecolor=blue,
  urlcolor=blue,
  bookmarksdepth=2
}

\numberwithin{equation}{section}  

\theoremstyle{plain} 
\newtheorem{theorem}{Theorem}[section]
\newtheorem{lemma}[theorem]{Lemma}
\newtheorem{proposition}[theorem]{Proposition}
\newtheorem{corollary}[theorem]{Corollary}

\theoremstyle{remark} 
\newtheorem{remark}[theorem]{Remark}

\title{Ruling out Infinite-Time Bubbling under $SO(4)$-Equivariant Yang--Mills Flow on $\mathbb{S}^4\times SU(2)$}
\author{Gabriel Tribe\thanks{University of Oxford, email: \texttt{gabriel.tribe@st-hughs.ox.ac.uk}.
Supported by the Heilbronn Institute for Mathematical Research and UKRI/EPSRC
through ``DA EPSRC Heilbronn Institute for Mathematical Research doctoral
training 2025'' (UKRI3009).}}
\date{October 9, 2026}

\begin{document}

\maketitle

 \begin{abstract}
We study the Yang--Mills flow on $\mathbb{S}^4$ restricted to
$\mathrm{SO}(4)$-equivariant connections on the trivial $\mathrm{SU}(2)$ bundle.
It was expected that antipodal instanton--anti-instanton configurations would lead to infinite-time singularities along the flow. We show that this does not occur. Moreover, for any smooth equivariant initial condition, the flow converges to the trivial flat connection. 
\end{abstract}

\setcounter{tocdepth}{2}
\tableofcontents
\bigskip

\section{Introduction}

\subsection{Background}

The Yang–-Mills flow is the negative $L^2$-gradient flow of the Yang–-Mills functional
$$
YM(A)=\frac 1 2\int_M |F_A|^2dV_g,
$$
whose critical points are Yang–-Mills connections. It was introduced by Atiyah and Bott as a parabolic approach to constructing Yang–-Mills connections and has become a central object of study in gauge theory and geometric analysis.

In dimension four the Yang--Mills functional is critical, the Palais--Smale
condition does not hold, and compactness may fail through bubbling.
Uhlenbeck compactness shows that a sequence of Yang--Mills connections with
uniformly bounded energy has, after passing to a subsequence and modulo gauge,
a smooth limit away from finitely many points of energy concentration.
Rescaling around these points produces bubbles, namely Yang--Mills connections on $\mathbb S^4$. Understanding the formation of such singular limits is therefore a natural problem in the study of Yang--Mills flow.

On the trivial $SU(2)$ bundle over $\mathbb S^4$, any bubble tree must have total charge zero. The simplest, lowest-energy bubbling configuration therefore consists of a charge $+1$ instanton together with a charge $-1$ anti-instanton. Among such configurations, the antipodal instanton--anti-instanton pair occupies a distinguished place, serving as a representative of the unique conformal class at the first bubbling threshold. It appears naturally in several constructions of non-minimal Yang--Mills connections; see~\cite{Parker1992NonMinimal,SibnerSibnerUhlenbeck1989,Wang1997MorseYM}. Motivated by this picture, it was expected that Yang–Mills flow started from a sufficiently concentrated antipodal instanton--anti-instanton pair would exhibit infinite-time bubbling. This expectation was explicitly stated in~\cite{Wang1997MorseYM,Sadun1994SymmetricFamily}. Our original aim was to establish such behaviour. Instead, we found that it does not occur. Every $SO(4)$-equivariant initial condition converges to the trivial connection.

From the PDE perspective the question of singularity formation in the Yang--Mills flow starts with Struwe \cite{struwe1994yangmills} where short-time existence for the Yang--Mills flow in dimension four was proved. Struwe constructed global weak solutions on four-manifolds, allowing for the possibility of finitely many singularities. Evidence that finite-time singularities should not occur came from the work of Schlatter et al. \cite{67fa8f92-32f2-39c3-bd3f-c407da60a7c7}, who proved global existence for the $SO(4)$--equivariant Yang--Mills flow on the four-ball. Their examples exhibited energy concentration occurring only in infinite-time, leading them to conjecture that Yang--Mills flow on closed four-manifolds should exist for all time. This conjecture was resolved by Waldron \cite{Waldron2019} in 2019, who proved long-time existence in the critical dimension four.

Nevertheless \emph{infinite-time} bubbling was known to occur; in Chapter 4 of Waldron's PhD thesis a radially symmetric example was constructed over the ball in $\mathbb{R}^4$. The bubbling comes from a mismatch in boundary conditions at $r=0$ and $r=1$. More recent work of  Sire et al. \cite{SireWeiZheng2022, sire2026globaloscillatorysolutionsyangmills} show infinite-time singularity formation on $\mathbb R^4$ in greater generality, as well as other exotic behaviour; we will discuss these results, and our relation to them, in subsection 1.3.

\subsection{Outline of Results}

In Section 2 we derive the reduction of the Yang--Mills flow under the $SO(4)$ cohomogeneity-one action on $\mathbb{S}^4$. Using Wang's theorem \cite{Wang1958} and a local computation for the singular orbits, we show that every smooth equivariant connection on the trivial $SU(2)$ bundle is determined by a single scalar function $a(\theta,t),\,\theta\in[0,\pi]$ satisfying the boundary behaviour of Proposition \ref{smoothextension}. The Yang--Mills flow then reduces to the scalar equation
$$
a_t
=
a_{\theta\theta}
+
\cot\theta\, a_\theta
-
\frac{4a(2a-1)(a-1)}{\sin^2\theta}.
$$
Ignoring boundary conditions, the standard BPST instanton and anti-instanton are explicit stationary solutions of this equation. We also make the surprising observation that the most natural smooth gluings of concentrated instantons and anti-instantons have Yang–Mills energy strictly below that of the ideal antipodal pair. Consequently, by energy monotonicity, these configurations cannot themselves converge to the ideal pair.

In Section 3 we prove global existence for all smooth equivariant initial data. Following ~\cite{67fa8f92-32f2-39c3-bd3f-c407da60a7c7,Waldron2014Thesis}, we remove the apparent singularities at the poles by introducing
$$
b(\theta,t)=\frac{a(\theta,t)}{\sin^2\theta},
$$
which transforms the equation into a semilinear parabolic equation on $\mathbb{S}^6$. We then construct explicit supersolutions built from glued instanton and anti-instanton profiles. These control the pole behaviour sufficiently to allow $\|b\|_\infty$ to blow up only in infinite-time, yielding long-time existence.

The remaining problem is to understand the long-time dynamics. Passing to coordinates
$$
z=\log\tan(\theta/2)\in\mathbb{R},
$$
the equation becomes
\begin{equation}\label{coshAC}
a_t=\cosh^2(z)\bigl(a_{zz}-4a(2a-1)(a-1)\bigr),
\end{equation}
which may be viewed as an Allen–Cahn equation with a spatially varying time scale. Under this correspondence, the instantons become heteroclinic fronts connecting the stable states $0$ and $1$. The main dynamical question is therefore analogous to the classical extinction-versus-propagation problem in bistable reaction–diffusion equations: do the two transition fronts eventually annihilate, or can they survive indefinitely while the intermediate plateau expands?

For the Allen–Cahn equation on $\mathbb R$, one naturally expects sufficiently decaying initial data converging to the same stable equilibrium at both spatial infinities to converge to that equilibrium as $t\to\infty$. Although the asymptotic theory of bistable reaction–diffusion equations is well developed (see, for example, Poláčik \cite{inbook}), we are unaware of a theorem establishing this conclusion for such general initial data, even in the classical Allen–Cahn setting. The relevant mechanism is metastability, first described by Carr–Pego \cite{carr1989metastable}: well-separated fronts interact only through exponentially small tails and therefore evolve on very long time scales, so solutions may remain close to multi-front configurations long before eventually converging to a homogeneous equilibrium. While such convergence appears to be the expected behaviour (see, for instance, the discussion in Alfaro et al. \cite{Alfaro_2024}), existing metastability theory typically assumes the solution is already close to a multi-front configuration. Our argument supplements these techniques with global comparison estimates, allowing us to deduce convergence from arbitrary initial data.

Section 4 analyses the flow away from the metastable regime. We construct a one-parameter family of decaying supersolutions which foliates the entire region below height $1$. Hence any solution lying strictly below $1$ converges uniformly to the trivial connection. We then prove a quantitative decay estimate for the supremum, which when combined with the pole growth from Section 3 shows that arbitrary solutions are eventually forced into a small neighbourhood of the metastable instanton--anti-instanton family.

Finally, in Section 5 we analyse the metastable regime, adapting the proof of \cite{carr1989metastable}. After introducing suitable modulation coordinates, we study the linearised operator about the metastable family and establish the existence of a distinguished small eigenvalue corresponding to the modulation direction. We then prove coercivity estimates on the orthogonal complement, allowing us to show that the error decays while the modulation parameter evolves only on a much slower timescale. This eventually forces the solution to pass below height $1$, at which point the comparison principle and the decaying supersolutions constructed earlier imply convergence to the trivial connection. We do not attempt to reproduce the full metastability framework developed for Allen--Cahn type equations; rather, we establish only the estimates required for the present problem.

In summary we prove the following.
\begin{theorem}\label{1.1}
Let $A_0$ be a smooth $SO(4)$-equivariant connection on the trivial $SU(2)$-bundle over $\mathbb{S}^4$, and let $A(t)$ denote the corresponding Yang--Mills flow with initial condition $A_0$.

Then $A(t)$ exists smoothly for all $t\ge0$, remains SO(4)-equivariant, and converges smoothly as $t\to\infty$ to the trivial flat connection. 

In particular, infinite-time bubbling does not occur within the class of smooth $SO(4)$-equivariant initial data.
\end{theorem}

\subsection{Comparisons and Outlook}
Recent work has shed new light on the first bubbling threshold for the trivial bundle on $S^4$. In the normalisation of \cite{dayaprema2026parabolicgaptheoremsyangmills}, a unit instanton has energy $4\pi^2$, so the ideal instanton--anti-instanton pair occurs at energy $8\pi^2$. Dayaprema and Waldron \cite{dayaprema2026parabolicgaptheoremsyangmills} proved that if
$$
YM(A_0)<4\pi^2(|\kappa(E)|+2),
$$
then Yang--Mills flow exists for all time and converges smoothly to an instanton. In particular, on the trivial bundle every flow with initial energy strictly below that of an instanton--anti-instanton pair converges to the flat connection. At the threshold itself, Waldron and Yin \cite{waldron2026yangmillsequationnearinstantonantiinstanton} showed that certain instanton--anti-instanton bubble trees cannot occur; in particular, a bubble tree containing a single self-dual unit-charge bubble and otherwise anti-self-dual bubbles is impossible.

Our result is complementary to these works. Together with \cite{waldron2026yangmillsequationnearinstantonantiinstanton}, it suggests unexpected rigidity at the first bubbling level for the trivial bundle. While Waldron and Yin establish that certain instanton--anti-instanton bubble configurations cannot arise as limits of sequences of Yang–Mills connections, we find a related rigidity at the level of the Yang–Mills flow, highlighting the possibility that the ideal antipodal configuration may likewise be inaccessible as an infinite-time singularity. More broadly, the relationship between bubble configurations, their sequential realisation, and their dynamical realisation under Yang–Mills flow remains largely unexplored and would be interesting to understand in greater generality. For example, suppose one could remove the symmetry assumptions in our result to show that no ideal instanton--anti-instanton pair can arise as an infinite-time singularity of Yang--Mills flow on the trivial bundle over $\mathbb S^4$. Such a result might extend the parabolic gap phenomenon beyond the energy range covered by \cite{dayaprema2026parabolicgaptheoremsyangmills}. In that case, the first non-minimal Yang--Mills connection constructed by Sibner, Sibner and Uhlenbeck \cite{SibnerSibnerUhlenbeck1989} would become the natural candidate for the first non-trivial asymptotic limit. Indeed, Waldron and Yin conjecture that the corresponding $2$-equivariant solution determines the only non-zero critical value below $16\pi^2$ (four times the instanton energy in their normalisation). It is noteworthy that the minimax construction of Sibner, Sibner and Uhlenbeck does not appear to extend to the $1$-equivariant setting corresponding to the antipodal pair. Their construction relies on the presence of a distinguished curvature direction in the basic $m$-instanton for $m>1$, which allows the construction of a non-trivial minimax loop below the bubbling threshold; no such direction exists when $m=1$, a feature that is suggestively consistent with the rigidity of the antipodal ideal pair observed above.

Regarding the removal of the symmetry assumptions, it is pertinent that Otto and Reznikoff \cite{OttoReznikoff2006} describe metastability in an abstract gradient-flow framework in terms of an energy--energy--dissipation inequality expressing coercivity transverse to a slow manifold, together with small variation of the energy along that manifold. This suggests a possible route to the non-equivariant problem, with the slow manifold given by well-separated instanton--anti-instanton configurations modulo gauge. The main analytic task would be to establish a \L ojasiewicz--Simon estimate at the bubble-tree boundary. This would extend estimates in the smooth case for Yang--Mills (see Appendix B of \cite{dayaprema2026parabolicgaptheoremsyangmills} for a modern account) and would be akin to results in the harmonic map setting such as \cite{Rupflin2024Lojasiewicz}.

We now discuss the relation with other literature on infinite-time singularity formation under Yang--Mills flow. In 2022 Sire et al. \cite{SireWeiZheng2022} were able to construct initial data bubbling at any given point in $\mathbb{R}^4$. Moreover, they showed the existence of bubble trees at infinite-time. More recently \cite{sire2026globaloscillatorysolutionsyangmills}, they uncovered an even richer range of asymptotic behaviours in the $SO(4)$-equivariant setting, including solutions whose scales oscillate indefinitely as $t\to\infty$. In the equivariant reduction considered in \cite{sire2026globaloscillatorysolutionsyangmills}, the flow is governed by
$$
u_t=u_{rr}+\frac1r u_r-\frac{2}{r^2}u(u-1)(u-2).
$$
Passing to logarithmic coordinates $z=\log r$ yields
$$
u_t=e^{-2z}\Bigl(u_{zz}-2u(u-1)(u-2)\Bigr),
$$
The resulting equation retains a distinguished non-compact end corresponding to spatial infinity, and the constructions in \cite{sire2026globaloscillatorysolutionsyangmills} exploit slowly decaying tails in that region to influence the long-time dynamics. By contrast, the compact geometry of $\mathbb{S}^4$ leads to the weighted equation \eqref{coshAC}, whose coefficients grow symmetrically at both ends. This reflects the absence of a spatial infinity in the present problem and suggests that the mechanism underlying the oscillatory solutions of \cite{sire2026globaloscillatorysolutionsyangmills} is not available in the closed setting.

At the same time, our convergence result should not be interpreted as evidence that infinite-time bubbling cannot occur on compact four-manifolds. Subsequent to the writing of this paper we found footnote eleven in Waldron's thesis \cite{Waldron2014Thesis} which, drawing on earlier constructions of Parker \cite{Parker1992NonMinimal}, proposed a singularity mechanism in the SU(2)-equivariant setting on $\mathbb S^4$. There, the reduced system admits asymptotic conditions which exclude the constant vacuum states and hence prevent convergence to the obvious smooth minimizers of the reduced problem. Waldron argues that the resulting mismatch is resolved through the formation of a non-trivial bubble tree, yielding a compact example of infinite-time singularity formation despite the existence of the trivial minimizer on the underlying bundle. By contrast, no such obstruction is present in the SO(4)-equivariant problem considered here: the trivial connection is compatible with the asymptotic data, and the flow converges to it.

The equivariant reduction viewpoint shows that different symmetry ans\"{a}tze lead to different reaction–diffusion systems exhibiting potentially radically different asymptotic behaviours. Waldron’s proposed singularity arises from asymptotic conditions which exclude the vacuum states of the reduced system. However, the reaction–diffusion literature suggests that such a mechanism for singularity formation is far from unique. In unbalanced bistable equations, the stable equilibria remain admissible, yet travelling fronts persist and invade one phase by another \cite{Fukao2004}. In more complicated systems, one encounters propagating terraces \cite{alma992983579365301631}, namely hierarchies of fronts evolving on distinct scales. Waldron’s bubble-tree picture appears closely analogous to this phenomenon. It would therefore be interesting to determine whether compact four-manifolds admit examples of infinite-time bubbling even when the reduced boundary conditions are compatible with a smooth minimizer.

More broadly, reaction–diffusion systems are known to exhibit a wide variety of exotic asymptotic behaviours at infinite-time, including oscillatory and non-quasiconvergent dynamics \cite{inbook}. This raises the possibility that suitably chosen compact equivariant reductions may also support genuinely exotic infinite-time behaviour.

\section{Geometric Setting}
\subsection{Cohomogeneity-one Ansatz}
Write the four-sphere as
$$
\mathbb{S}^4
=
\Bigl\{
(\cos\theta,\sin\theta\,\omega)
\in \mathbb{R}\times\mathbb{R}^4
\;\Big|\;
\theta\in[0,\pi],\;
\omega\in \mathbb{S}^3
\Bigr\},
$$
equipped with the metric
$$
g=d\theta^2+\sin^2\theta\,g_{\mathbb S^3}.
$$
Define
$$
\gamma:SO(4)\times\mathbb S^4\longrightarrow\mathbb S^4
$$
by
$$
\gamma\bigl(B,(\cos\theta,\sin\theta\,\omega)\bigr)
=
(\cos\theta,\sin\theta\,B\omega).
$$
For each $B\in SO(4)$, write $\gamma_B$ for the associated
diffeomorphism. The action preserves the coordinate $\theta$ and acts
transitively on the $\mathbb S^3$-factor. In particular, each level set
$\{\theta=\mathrm{const}\}$ is an $SO(4)$-orbit. For
$\theta\in(0,\pi)$, the stabiliser of
$$
p=(\cos\theta,\sin\theta\,\omega)
$$
is
$$
\{B\in SO(4)\mid B\omega=\omega\}\cong SO(3).
$$
Such orbits are called \emph{principal} and are diffeomorphic to
$SO(4)/SO(3)\simeq\mathbb S^3$. At the endpoints
$\theta=0,\pi$, the $\mathbb S^3$-factor collapses to a point, so the
stabiliser is the entire group $SO(4)$. These give rise to two
\emph{singular} orbits. We thus obtain a decomposition of
$\mathbb S^4$ into $SO(4)$-orbits, with orbit space naturally
identified with the interval $[0,\pi]$.

Our aim is to study equivariant connections on the trivial
principal $SU(2)$-bundle over $\mathbb S^4$. It is convenient to
describe the equivariant structure using the double cover
$$
\rho:
\operatorname{Spin}(4)
=
SU(2)_L\times SU(2)_R
\longrightarrow SO(4).
$$
Identifying $\mathbb R^4$ with the quaternions $\mathbb H$, the lifted
action on the base is
$$
(q_L,q_R)\cdot
(\cos\theta,\sin\theta\,\omega)
=
\bigl(
\cos\theta,
\sin\theta\,q_L\omega q_R^{-1}
\bigr).
$$
This action factors through $\rho$ and hence induces the standard
$SO(4)$-action described above.
For $\theta\in(0,\pi)$, the principal isotropy subgroup is the diagonal
subgroup
$$
\widetilde H
=
\{(q,q)\mid q\in SU(2)\}
\simeq SU(2),
$$
whereas the singular isotropy subgroup at either pole is the whole group
$$
\widetilde H_-=\widetilde H_+
=
SU(2)_L\times SU(2)_R.
$$
We briefly recall the equivariant bundle construction. Let $G$ act
transitively on $G/H$, and suppose that a principal $K$-bundle over
$G/H$ carries a lift of the $G$-action commuting with the principal
right $K$-action. After choosing a point in the fibre over the identity
coset, one obtains an isotropy homomorphism
$$
\lambda:H\longrightarrow K,
$$
well defined up to conjugation. By Wang's theorem \cite{Wang1958}, the
equivariant principal $K$-bundles over $G/H$ are classified by
conjugacy classes of such homomorphisms.
In the cohomogeneity-one setting, an equivariant bundle is specified by
singular isotropy homomorphisms
$$
\lambda_-:\widetilde H_-\longrightarrow K,
\qquad
\lambda_+:\widetilde H_+\longrightarrow K,
$$
whose restrictions to the principal isotropy subgroup agree. For the
trivial $SU(2)$-bundle considered here, we choose the same projection
at the two singular orbits:
\begin{equation}\label{singularisotropy}
\widetilde\lambda_-
=
\widetilde\lambda_+
=
\operatorname{pr}_R,
\qquad
\widetilde\lambda_\pm(q_L,q_R)=q_R.
\end{equation}
Their common restriction to $\widetilde H$ is
$$
\widetilde\lambda_H:\widetilde H\longrightarrow SU(2),
\qquad
\widetilde\lambda_H(q,q)=q.
$$

Let us verify explicitly that this choice gives the trivial principal bundle.
The homogeneous bundle over a principal orbit is
$$
\operatorname{Spin}(4)\times_{\widetilde H}SU(2),
$$
where
$$
(g,k)\sim
\bigl(gh,\widetilde\lambda_H(h)^{-1}k\bigr).
$$
Each extension $\widetilde\lambda_\pm$ determines a trivialisation over
the common boundary of the two singular neighbourhoods by
$$
[g,k]\longrightarrow
\bigl(g\widetilde H,\widetilde\lambda_\pm(g)k\bigr).
$$
The transition function between these trivialisations is therefore
$$
c:\operatorname{Spin}(4)/\widetilde H\longrightarrow SU(2),
\qquad
c(g\widetilde H)
=
\widetilde\lambda_+(g)
\widetilde\lambda_-(g)^{-1}.
$$
Since the two singular isotropy homomorphisms in
\eqref{singularisotropy} agree, one has $c\equiv 1$. Thus the clutching map is null-homotopic and the resulting principal
$SU(2)$-bundle has $c_2(P)=0$. It is therefore isomorphic to the trivial bundle
$$
P:=\mathbb S^4\times SU(2).
$$

We now describe the class of equivariant connections that will be studied
throughout the remainder of the paper. Strictly speaking, the equivariant structure constructed above is a lift of
the $\operatorname{Spin}(4)$-action on $\mathbb S^4$, rather than an honest
$SO(4)$-action on the principal bundle. Indeed, the nontrivial element of
$$
\ker\!\bigl(\operatorname{Spin}(4)\rightarrow SO(4)\bigr)
=
\{(1,1),(-1,-1)\}
$$
acts trivially on the base but by multiplication by the central element
$-1\in SU(2)$ on each fibre. Thus the lifted action does not descend to a
genuine $SO(4)$-action on $P$. This distinction disappears at the level of connections i.e. the induced action of
$\operatorname{Spin}(4)$ on the space of connections factors through
$SO(4)$. Indeed, since
$$
\operatorname{Ad}_{-1}
=
\mathrm{id}_{\mathfrak{su}(2)},
$$
the constant gauge transformation determined by $-1$ fixes every
connection. Throughout this paper, we shall therefore call a connection
\emph{$SO(4)$-equivariant} if it is invariant under the lifted
$\operatorname{Spin}(4)$-action. Equivalently, if
$\widetilde B=(q_L,q_R)\in\operatorname{Spin}(4)$ is any lift of
$B\in SO(4)$, then
\begin{equation}
A(p,v)
=
\operatorname{Ad}_{q_R^{-1}}
\Bigl(
A\bigl(
\gamma_{\widetilde B}(p),
d\gamma_{\widetilde B}(p)v
\bigr)
\Bigr).
\end{equation}

The symmetry assumption imposes strong restrictions on the form of an
equivariant connection. To make these explicit, we decompose the
connection into its radial and orbital components with respect to the
cohomogeneity-one decomposition of $\mathbb S^4$. Using $\mathbb S^4\setminus\{\text{poles}\}
\simeq
(0,\pi)\times\mathbb S^3$ write
$$
A(\theta,\omega)
=
A_\theta(\theta,\omega)\,d\theta
+
A_{\mathrm{orb}}(\theta,\omega),
$$
where
$$
A_\theta(\theta,\cdot)
\in
\Gamma(\operatorname{ad}P),
\qquad
A_{\mathrm{orb}}(\theta,\cdot)
\in
\Omega^1(\mathbb S^3,\mathfrak{su}(2)).
$$
\begin{lemma}
    Let $A$ be an $SO(4)$-equivariant connection on $P$. Then for every $\theta\in(0,\pi)$ the radial component
    vanishes on the principal orbit:
    $$
    A_\theta(\theta,\omega)\equiv 0\qquad\text{for all }\omega\in \mathbb{S}^3.
    $$
\end{lemma}

\begin{proof}
    Fix $\theta\in(0,\pi)$ and choose a reference point
    $\omega_0\in \mathbb{S}^3$. As $SO(4)$ acts transitively on the principal
    orbit $\mathbb{S}^3$, for any $\omega\in \mathbb{S}^3$ there exists
    $B\in SO(4)$ with $\gamma_B(\theta,\omega_0)=(\theta,\omega)$.
    Let $\widetilde B=(q_L,q_R)\in\operatorname{Spin}(4)$ be any lift of $B$.
    Equivariance of $A$ applied to the vector
    $\partial_\theta$ gives
    $$
    A_\theta(\theta,\omega)
    =
    \operatorname{Ad}_{q_R^{-1}}\bigl(A_\theta(\theta,\omega_0)\bigr),
    $$
    since $d\gamma_B(\partial_\theta)=\partial_\theta$.

    If $h\in H_0:=\mathrm{Stab}_{SO(4)}(\omega_0)$, then $B$ and $Bh$
    both map $(\theta,\omega_0)$ to $(\theta,\omega)$. Lift $h$ to
    $\widetilde h\in\widetilde H_0:=\rho^{-1}(H_0)\subset
    SU(2)_L\times SU(2)_R$. Comparing the two expressions above for
    $\widetilde B$ and for $\widetilde B\widetilde h$ shows that
    $$
    \operatorname{Ad}_{\widetilde\lambda(\widetilde h)^{-1}}
    \bigl(A_\theta(\theta,\omega_0)\bigr)
    =
    A_\theta(\theta,\omega_0)
    \qquad\text{for all }\widetilde h\in\widetilde H_0.
    $$

    Thus $A_\theta(\theta,\omega_0)$ is fixed by the adjoint action of
    $\widetilde\lambda(\widetilde H)$. Since
    $$
    \widetilde\lambda:\widetilde H\longrightarrow SU(2),
    \qquad
    (q,q)\longmapsto q,
    $$
    is surjective, $A_\theta(\theta,\omega_0)$ is fixed by
    $\operatorname{Ad}_{SU(2)}$. Hence $A_\theta(\theta,\omega_0)$ is in the centre of $\frak{su}(2)$ which is trivial. The equivariance relation above
    implies $A_\theta(\theta,\omega)=0$ for all
    $\omega\in\mathbb S^3$.
\end{proof}

\noindent
We now describe the orbital part $A_{\mathrm{orb}}$ using the following theorem, written at the level of the double cover.

\begin{theorem}[Wang's Theorem \cite{Wang1958}]
Decompose
$$
\mathfrak{spin}(4)=\widetilde{\mathfrak h}\oplus\mathfrak m
$$
with $\mathfrak m$ an
$\operatorname{Ad}(\widetilde H)$-invariant complement of
$\widetilde{\mathfrak h}$. Then the orbital parts
$A_{\mathrm{orb}}$ of $SO(4)$-equivariant connections $A$ on
$P$ are in one-to-one correspondence with linear maps
$$
M:\mathfrak m\longrightarrow\mathfrak{su}(2)
$$
satisfying the intertwining condition
\begin{equation}\label{intertwiner}
M\bigl(\operatorname{Ad}_{\widetilde h}X\bigr)
=
\operatorname{Ad}_{\widetilde\lambda(\widetilde h)}
\bigl(M(X)\bigr)
\qquad
\forall\,\widetilde h\in\widetilde H,\;
X\in\mathfrak m.
\end{equation}
\end{theorem}

\noindent
We denote by
$$
\operatorname{Hom}_{\widetilde H}
\bigl(\mathfrak m,\mathfrak{su}(2)\bigr)
$$
the vector space of $\widetilde H$-equivariant linear maps
$M:\mathfrak m\to\mathfrak{su}(2)$ satisfying
\eqref{intertwiner}. To make this concrete choose an orthonormal basis $\{e_i\}_{i=1}^3$
of $\mathfrak m$ and a left-invariant coframe $\{\eta^i\}_{i=1}^3$
on the unit $\mathbb S^3$ satisfying
\begin{equation}
d\eta^i=-\epsilon_{ijk}\eta^j\wedge\eta^k,
\qquad
\epsilon_{123}=+1.
\end{equation}
Let $\{T_i\}_{i=1}^3$ be a basis of $\mathfrak{su}(2)$ with
$[T_i,T_j]=2\epsilon_{ijk}T_k$ and
$\langle T_i,T_j\rangle=4\delta_{ij}$\footnote{This normalisation corresponds to taking Hamilton's quaternionic generators and scaling the Lie algebra inner product so that the energy of an instanton is $8\pi^2$.}. Let
$$
\eta:=\operatorname{Im}(\overline{\omega}\,d\omega)
$$
be the $\mathfrak{su}(2)$-valued left Maurer--Cartan form on
$SU(2)\cong\mathbb S^3$. With respect to the basis
$\{T_i\}_{i=1}^3$, this is precisely
$$
\eta=\sum_{i=1}^3T_i\,\eta^i.
$$ 

Under the identification of $\widetilde H$ with $SU(2)$, the modules
$\mathfrak m$ and $\mathfrak{su}(2)$ are isomorphic as real
$\widetilde H$-representations (both realise the
three-dimensional adjoint representation). By the real version of
Schur's lemma, the space
$$
\operatorname{Hom}_{\widetilde H}
(\mathfrak m,\mathfrak{su}(2))
$$
is one-dimensional. Via the bases $\{e_i\}$ and
$\{T_i\}$, choose the intertwiner sending
$e_i$ to $T_i$. Every element of
$\operatorname{Hom}_{\widetilde H}
(\mathfrak m,\mathfrak{su}(2))$
is then a scalar multiple of this map. Therefore there exists a scalar
function $a(\theta)$ such that, for each
$\theta\in(0,\pi)$,
\begin{equation}\label{ansatz}
A_{\mathrm{orb}}(\theta,\omega)
=
a(\theta)\sum_{i=1}^3T_i\,\eta^i(\omega).
\end{equation}
In other words, any $SO(4)$-equivariant connection is determined by a
single scalar function $a(\theta)$ on the orbit space $[0,\pi]$.

\subsection{Smooth Extension to the Singular Orbits}

We now determine the conditions under which the ansatz
\eqref{ansatz} extends smoothly across the singular orbits.

Fix the north pole and introduce the normal coordinate
$$
x=\sin\theta\,\omega\in\mathbb H,
\qquad
r=|x|=\sin\theta.
$$
Writing $x=r\omega$,
$$
dx=dr\,\omega+r\,d\omega,
$$
and therefore
$$
\overline x\,dx
=
r\,dr+r^2\overline\omega\,d\omega.
$$
Taking imaginary parts gives
$$
\operatorname{Im}(\overline x\,dx)
=
r^2\eta.
$$
Hence
$$
A
=
\frac{a(\theta)}{\sin^2\theta}
\operatorname{Im}(\overline x\,dx).
$$

Since
$\operatorname{Im}(\overline x\,dx)$ is a smooth polynomial
$\mathfrak{su}(2)$-valued one-form on $\mathbb R^4$, the connection
extends smoothly across the north pole if and only if the coefficient function 
$$
b(\theta):=\frac{a(\theta)}{\sin^2\theta}
$$
extends smoothly as a radial function of
$x$. The same computation applies in normal coordinates about the south
pole. It is a classical fact that a radial function on
$\mathbb R^4$ extends smoothly across the origin if and only if it is a
smooth function of $r^2=|x|^2$. Indeed, every smooth radial function is O(4)-invariant, its Taylor expansion consists only of O(4)-invariant homogeneous polynomials, which are precisely the powers of $|x|^2$. Since $r=\sin\theta$, this translates
directly into the following condition.

\begin{proposition}\label{smoothextension}
The equivariant connection
$$
A
=
a(\theta)\sum_{i=1}^3T_i\eta^i
$$
extends smoothly across both singular orbits if and only if $b$ extends
smoothly to an even function of $\theta$ near $0$ and to an even
function of $\pi-\theta$ near $\pi$.
\end{proposition}

\subsection{Equivariant Yang--Mills Derivation}
To justify the reduction, note that the Yang–Mills functional is invariant under the $\operatorname{Spin}(4)$-action, and the action is isometric with respect to the $L^2$ metric. Therefore its gradient is equivariant. The tangent space to the manifold of equivariant connections consists precisely of equivariant $\operatorname{ad}P$-valued one-forms. Hence critical points of the restricted functional are Yang--Mills connections, and the gradient flow of the restricted functional is precisely the restriction of the Yang--Mills flow.

We now plug in our symmetric ansatz. As a first step in the computation we have
$$
dA=a'\sum_{i=1}^3T_id\theta\wedge \eta^i-a\sum_{i=1}^3\epsilon_{ijk}\eta^j\wedge\eta^kT_i.
$$
Next,
$$
[A,A]=a^2\sum_{j,k=1}^3[T_j,T_k]\eta^j\wedge \eta^k=2a^2\sum_{j,k=1}^3\epsilon_{jki}T_i\eta^j\wedge\eta^k.
$$
Combining we have 
$$
F_{0i}=\frac{a'}{\sin\theta}T_i\quad\text{ and }\quad F_{jk}=\frac{2(a^2-a)}{\sin^2\theta}\sum_{i=1}^3\epsilon_{ijk}T_i,
$$
where we have computed in the orthonormal basis $\{\sin\theta\,\eta^i\}$. Now
$$
|F_A|^2=\sum_{i=1}^3\langle F_{0i},F_{0i}\rangle+\sum_{1\leq j<k\leq3}\langle F_{jk},F_{jk}\rangle=\frac{12(a')^2}{\sin^2\theta}+\frac{48\bigl(a^2-a\bigr)^2}{\sin^4\theta},
$$
hence
$$
YM(A)=\frac{1}{2}\int_{\mathbb{S}^4}|F_A|^2=6\mathrm{vol}(\mathbb{S}^3)\int_0^\pi(a')^2\sin\theta+\frac{4(a^2-a)^2}{\sin\theta}d\theta.
$$
Let $\phi$ be a smooth variation of $a$; a direct computation gives
\begin{align*}
    \frac{d}{d\varepsilon}\Big|_{\varepsilon=0}
YM(a+\varepsilon\phi)
&=
C\int_0^\pi
\left(
2a'\phi' \sin\theta
+ \frac{8(a^2-a)(2a-1)\phi}{\sin\theta}
\right)
d\theta\\
&=
C\int_0^\pi
\left(
-2(a''+a'\cot\theta)
+ \frac{8(a^2-a)(2a-1)}{\sin^2\theta}
\right)
\phi\, \sin\theta\, d\theta.
\end{align*}
Since $\sin\theta>0$ on $(0,\pi)$, the stationary equation is
\begin{equation}
    a''+a'\cot\theta
- \frac{4a(2a-1)(a-1)}{\sin^2\theta}
=0.    
\end{equation}
We will use the notation 
$$
L(a)=a''+a'\cot\theta-\frac{4a(2a-1)(a-1)}{\sin^2\theta}
$$
for the reduced Yang--Mills operator and 
\begin{equation}\label{nonlinpolyeq}
f(a)=-4a(2a-1)(a-1)
\end{equation}
for the polynomial in the nonlinearity.  

\medskip

We now show how the $L^2$
inner product reduces under the equivariant ansatz. Let $\dot A$ be an equivariant variation:
$$
\dot A(\theta)
=
\phi(\theta)\sum_{i=1}^3 T_i \eta_i.
$$
Using the orthonormal coframe $\{\, d\theta,\ \sin\theta\,\eta_1,\ \sin\theta\,\eta_2,\ \sin\theta\,\eta_3 \,\}$, we have
$$
|\dot A|^2
=
\frac{12\,\phi(\theta)^2}{\sin^2\theta}.
$$
The volume form on $\mathbb{S}^4$ is $d\mathrm{vol}_{\mathbb{S}^4}
=
\sin^3\theta\, d\theta\, d\mathrm{vol}_{\mathbb{S}^3}$. Therefore
$$
\|\dot A\|_{L^2(\mathbb{S}^4)}^2
=
\int_{\mathbb{S}^4} |\dot A|^2
=
12\,\mathrm{vol}(\mathbb{S}^3)
\int_0^\pi
\phi(\theta)^2 \sin\theta\, d\theta.
$$
Thus, up to a constant, the $L^2$ inner product on
equivariant variations reduces to
$$
\langle \phi,\psi\rangle
=
\int_0^\pi
\phi(\theta)\psi(\theta)\sin\theta\, d\theta.
$$
Hence the Yang--Mills flow reduces to the
$L^2(\sin\theta\, d\theta)$–gradient flow of the one–dimensional
functional $YM(a)$:
\begin{equation}\label{thePDE}
a_t
=
a''+a'\cot\theta
- \frac{4a(2a-1)(a-1)}{\sin^2\theta}.    
\end{equation}

\subsection{Initial Conditions}\label{initconditsec}

For each $\lambda>0$ the standard anti-self-dual BPST instanton on
$\mathbb{R}^{4}$, when lifted by stereographic projection centred at the
north pole, is
$$
a^\mathrm{ASD}_\lambda(\theta)
  = \frac{\tan^{2}(\theta/2)}{\lambda^{2}+\tan^{2}(\theta/2)}.
$$
These profiles satisfy the first-order anti-self-duality equation
\begin{equation}\label{ASDeqn}
a'(\theta)
  = -\frac{2a(a-1)}{\sin\theta},
\end{equation}
and hence the Yang--Mills equation.

Let $\xi\in\mathbb{R}^{4}$ denote the coordinates on the northern chart of
$\mathbb{S}^{4}$ obtained by stereographic projection from the south pole,
and let $\zeta\in\mathbb{R}^{4}$ denote the coordinates on the southern chart
obtained by stereographic projection from the north pole. These are related by
the inversion formula
$$
\zeta=\frac{\xi}{|\xi|^{2}},
\qquad
\xi=\frac{\zeta}{|\zeta|^{2}}.
$$
As $\theta\to\pi$ the profile satisfies
$a^\mathrm{ASD}_\lambda(\theta)\to1$, so that in the northern chart the
corresponding connection approaches a pure gauge
$$
A \sim g^{-1}dg,
\qquad
g(\xi)=\frac{\xi}{|\xi|}\in SU(2),
$$
up to terms that decay rapidly in $|\xi|$. Although we will not need it here, note that, by Uhlenbeck's removable-singularity theorem, the connection extends smoothly across the south pole, yielding the BPST instanton on the
charge $1$ $SU(2)$ bundle over $\mathbb{S}^{4}$.

The parameter $\lambda$ is the scale, governing the concentration of
curvature near the north pole. For each $\lambda\in(0,1)$,
$$
\frac{a^\mathrm{ASD}_\lambda(\theta)}{\sin^2\theta}
=
\frac{1}{4\lambda^2}+O(\theta^2)
$$
as $\theta\to0$. Plotting the solutions, see Figure~1, we see that as
$\lambda\downarrow0$ the transition from $0$ to $1$ occurs in an
increasingly narrow region near $\theta=0$ and the curvature becomes more
strongly concentrated at the north pole. In the limit the connections
converge to the ideal anti-self-dual configuration, corresponding to
$a\equiv1$ on $(0,\pi]$.

\begin{figure}[h]
\centering
\includegraphics[width=0.6\textwidth]{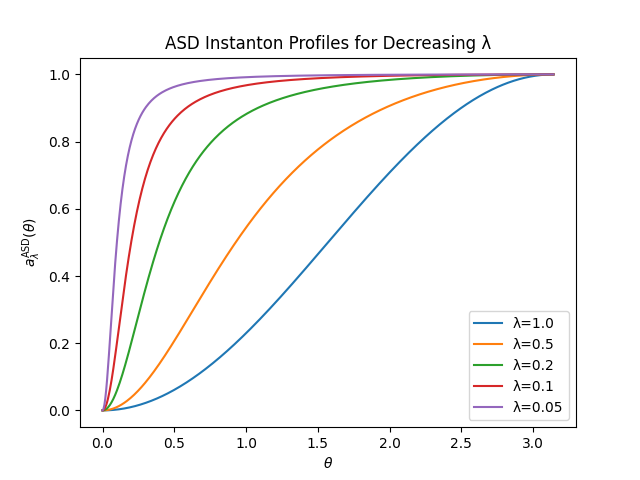}
\caption{}
\end{figure}

There exists an analogous family of Yang--Mills connections which are
self-dual, meaning that they satisfy \eqref{ASDeqn} with a positive sign
on the right-hand side. These may be obtained geometrically from the
anti-self-dual instantons by reflection across the equator.
Define
$$
\iota:S^4\longrightarrow S^4,
\qquad
\iota(\cos\theta,\sin\theta\,\omega)
=
(-\cos\theta,\sin\theta\,\omega)
=
(\cos(\pi-\theta),\sin(\pi-\theta)\omega).
$$
The map $\iota$ is an orientation-reversing isometry which interchanges
the north and south poles and commutes with the $SO(4)$-action.

\begin{lemma}
Let $A^{\mathrm{ASD}}_\lambda
=
a^{\mathrm{ASD}}_\lambda(\theta)
\sum_{i=1}^3 T_i\eta^i$. Then $\iota^*A^{\mathrm{ASD}}_\lambda$ is a self-dual instanton centred
at the south pole, and in the equivariant ansatz it is given by
$$
\iota^*A^{\mathrm{ASD}}_\lambda
=
a^{\mathrm{SD}}_\lambda(\theta)
\sum_{i=1}^3T_i\eta^i =
a^{\mathrm{ASD}}_\lambda(\pi-\theta)\sum_{i=1}^3T_i\eta^i
=
\frac{1}{1+\lambda^2\tan^2(\theta/2)}\sum_{i=1}^3T_i\eta^i.
$$
\end{lemma}

\begin{proof}
Since $\iota$ leaves the variable $\omega$ unchanged, one has $\iota^*\eta^i=\eta^i$. Therefore
$$
\iota^*A^{\mathrm{ASD}}_\lambda
=
a^{\mathrm{ASD}}_\lambda(\pi-\theta)
\sum_{i=1}^3T_i\eta^i.
$$
It remains to check the duality type. Curvature is natural under
pullback: $F_{\iota^*A^{\mathrm{ASD}}_\lambda}
=
\iota^*F_{A^{\mathrm{ASD}}_\lambda}$. Because $\iota$ is an orientation-reversing isometry, its pullback
anticommutes with the Hodge star on two-forms: $*\iota^*\alpha=-\iota^*(*\alpha)$, and hence
\begin{align*}
*F_{\iota^*A^{\mathrm{ASD}}_\lambda}
&=
*\iota^*F_{A^{\mathrm{ASD}}_\lambda}\\
&=
-\iota^*\bigl(*F_{A^{\mathrm{ASD}}_\lambda}\bigr)\\
&=
\iota^*F_{A^{\mathrm{ASD}}_\lambda}\\
&=
F_{\iota^*A^{\mathrm{ASD}}_\lambda}.
\end{align*}
Thus $\iota^*A^{\mathrm{ASD}}_\lambda$ is self-dual. 
\end{proof}

Using the instanton and anti-instanton connections described above,
one may construct smooth equivariant initial data which resemble a
concentrated instanton at one pole and a concentrated anti-instanton at
the other. For example, one may interpolate between the profiles using a
cutoff function supported in an equatorial collar, or consider the
additive superposition
$$
a^{\mathrm{ASD}}_\lambda+a^{\mathrm{SD}}_\lambda-1.
$$
It was configurations of this type that were expected to lead to
infinite-time bubbling. Surprisingly, a direct computation shows that the
additive gluing has Yang--Mills energy strictly less than that of the
ideal antipodal pair:
$$
YM(a^{\mathrm{ASD}}_\lambda+a^{\mathrm{SD}}_\lambda-1)
<
16\pi^2.
$$
The same phenomenon occurs for the most natural cutoff constructions. Since the Yang--Mills flow decreases energy monotonically, such initial conditions cannot converge to the
ideal antipodal instanton--anti-instanton pair.
Nevertheless, one can also construct less natural gluings whose energy is strictly greater than $16\pi^2$
while still concentrating onto the same ideal pair. For example, if
$\psi$ is a bump function satisfying $\psi\equiv1$ on
$[\delta,\pi-\delta]$, then
$$
YM\!\left(
a^{\mathrm{ASD}}_\lambda
+
a^{\mathrm{SD}}_\lambda
-
1
+
\lambda\psi
\right)
\downarrow 16\pi^2
\qquad
\text{as }\lambda\downarrow0,
$$
for $\delta$ sufficiently small.

\section{Long-Time Existence}

\subsection{Regularising and Short Time Existence}
The reduced form of the Yang--Mills flow \eqref{thePDE} does not immediately seem amenable to standard parabolic theory owing to the coefficients becoming singular at $\theta=0,\pi$. Following the approach of \cite{67fa8f92-32f2-39c3-bd3f-c407da60a7c7}, \cite{Waldron2014Thesis} we can desingularise the nonlinearity via the change of variables $b(\theta,t):=\frac{a(\theta,t)}{\sin^2\theta}$. A direct computation shows that \eqref{thePDE} is equivalent to
\begin{equation}\label{beqn}
\partial_t b
=
\partial_{\theta\theta}b
+
5\cot\theta\,\partial_\theta b
-6b
+12b^2
-8\sin^2\theta\,b^3.
\end{equation}
Let $\mathbb{S}^6$ be equipped with its round metric and fix a pole $o\in\mathbb{S}^6$. We identify a function of $\theta\in[0,\pi]$ with the corresponding zonal function on $\mathbb{S}^6$, that is, the function depending only on the geodesic distance $\theta(x)=d_{\mathbb{S}^6}(x,o)$.

\begin{lemma}\label{shrttime}
Let $b_0\in L^\infty(\mathbb{S}^6)$ be zonal. Then there exists a maximal time
$T_{\max}\in(0,\infty]$ and a unique maximal mild solution $b$ of \eqref{beqn} with initial data $b_0$, which is smooth for $t>0$. Moreover, if $T_{\max}<\infty$, then
$$
\limsup_{t\uparrow T_{\max}}
\|b(\cdot,t)\|_{L^\infty(\mathbb{S}^6)}
=
\infty.
$$
\end{lemma}

\begin{proof}
If $u(x)=f(\theta(x))$ is a zonal function on $\mathbb{S}^6$, then
$$
\Delta_{\mathbb{S}^6}u
=
f''(\theta)+5\cot(\theta)f'(\theta).
$$
Hence \eqref{beqn} may be viewed as a semilinear heat equation on
$\mathbb{S}^6$ restricted to zonal functions:
$$
\partial_t b
=
\Delta_{\mathbb{S}^6} b -6b+12b^2-8\sin^2\theta\,b^3.
$$
The class of zonal functions is preserved by both the Laplacian and the
nonlinearity. Moreover, although $\theta$ itself is not smooth on
$\mathbb{S}^6$, the coefficient $\sin^2\theta$ is smooth; indeed, in ambient
coordinates on $\mathbb{S}^6\subset\mathbb{R}^7$ one has $\sin^2\theta = 1-x_7^2$. Consequently the nonlinearity is smooth on the compact manifold
$\mathbb{S}^6$. Standard semilinear parabolic theory therefore yields a unique maximal mild solution, which is smooth for $t>0$, together with the stated $L^\infty$ continuation criterion.
\end{proof}

If $a_0$ arises from a smooth equivariant connection, Proposition~\ref{smoothextension} implies that $b_0:=\frac{a_0}{\sin^2\theta}$ extends to a smooth even function at both endpoints, and hence to a smooth zonal function on $\mathbb{S}^6$. Lemma \ref{shrttime} applies, in this case $b$ is smooth up to $t=0$, and
$$
a(\theta,t)=\sin^2\theta\,b(\theta,t)
$$
gives the unique smooth solution of \eqref{thePDE} with initial data $a_0$, satisfying Proposition~\ref{smoothextension} for every $t<T_{\max}$.

\begin{remark}
    The standard parabolic comparison principle applies to \eqref{beqn}. Throughout the remainder of the paper we apply comparison only to sub- and supersolutions a for which $a/\sin^2\theta$ is bounded (equivalently, the associated b-functions are bounded on $\mathbb S^6$). Such comparisons are understood via the regularised equation \eqref{beqn}, and we shall freely state the resulting inequalities in terms of the corresponding functions a.
\end{remark}

\subsection{A Degenerating Global Supersolution}

The result of Lemma \ref{shrttime} aligns well with the intuition gained from Figure 1: maximal existence time of solutions should be governed by the quadratic behaviour at the poles. In order to control said behaviour we will build a supersolution from the (anti-)self-dual instantons described previously. In particular, the profile $a^{\mathrm{ASD}} + a^{\mathrm{SD}} - 1$ considered in subsection \ref{initconditsec} will be of great importance. We denote it
$$
u_\lambda(\theta)
:=
\frac{\tan^2(\theta/2)}{\lambda^2+\tan^2(\theta/2)}
+
\frac{1}{1+\lambda^2\tan^2(\theta/2)} - 1.
$$
At $\lambda=1$ it is identically $0$ and as $\lambda\uparrow1$ it converges locally uniformly on $(0,\pi)$ to $0$. For $\lambda\in(0,1)$ it attains $0$ only at $\theta=0,\pi$ and otherwise lies in $(0,1)$. We will use the notation
$$
r=r(\theta):=\tan(\theta/2).
$$
\begin{lemma}\label{Fcomp}
Fix $k>1$. Define
$$
w(\theta):=u_\lambda(\theta)+k u_\mu(\theta).
$$
For $\mu$ sufficiently small, and then $\lambda$ sufficiently small,
$$
\inf_{\theta\in(0,\pi)}\frac{L(w(\theta))}{\partial_\lambda u_\lambda(\theta)}=\lim_{\theta\to 0,\pi}\frac{L(w(\theta))}{\partial_\lambda u_\lambda(\theta)}
=
-3k\left(\frac{1}{\mu^2}-\mu^2\right)\lambda + O(\lambda^3).
$$
\end{lemma}
\begin{proof}
Set $y=r^2+r^{-2}$. Then $y$ is invariant under $\theta\mapsto\pi-\theta$, decreases from
$\infty$ to $2$ on $(0,\pi/2)$, and
$$
u_\lambda(\theta)
=
\frac{1-\lambda^4}{1+\lambda^4+\lambda^2y},
\qquad
u_\mu(\theta)
=
\frac{1-\mu^4}{1+\mu^4+\mu^2y}.
$$
Write $\Phi(y):=\frac{L(w)}{\partial_\lambda u_\lambda}$. A direct differentiation and simplification gives
$$
-\partial_y\Phi(y)
=
\frac{P_{\lambda,\mu,k}(y)}
{2\lambda\,
 (1+\lambda^4+\lambda^2y)^2
 (1+\mu^4+\mu^2y)^4
 \bigl((1+\lambda^4)y+4\lambda^2\bigr)^2},
$$
where
$$
P_{\lambda,\mu,k}(y)=\sum_{j=0}^{6}c_j(\lambda,\mu,k)y^j
$$
is a polynomial. If $(m_0,\ldots,m_6)=(0,0,0,0,2,4,6),
\,
(n_0,\ldots,n_6)=(0,2,2,4,4,4,4),$ then
$$
\lim_{\mu\downarrow0}\lim_{\lambda\downarrow0}
\lambda^{-2m_j}\mu^{-2n_j}c_j(\lambda,\mu,k)
=
\gamma_j(k),
$$
where
$$
\footnotesize
\begin{array}{c|ccccccc}
j
&0&1&2&3&4&5&6\\ \hline
\gamma_j(k)
&2k(k+1)(2k+1)
&8k(k+1)(2k+1)
&6k^2(k+1)
&6k(k+1)
&12k(k+2)
&6k(k+5)
&12k .
\end{array}
$$
All these quantities are strictly positive. Hence, after choosing $\mu$
and then $\lambda$ sufficiently small, every coefficient $c_j(\lambda,\mu,k)$ is
positive. Thus $P_{\lambda,\mu,k}(y)>0$ for $y\geq2$, and consequently $\partial_y\Phi(y)<0$. Since $y$ decreases on $(0,\pi/2)$, the function
$$
\theta\longmapsto
\frac{L(w(\theta))}{\partial_\lambda u_\lambda(\theta)}
$$
is increasing there. It is symmetric under
$\theta\mapsto\pi-\theta$, so its infimum is attained in the limits
$\theta\to0,\pi$.

It remains to compute this limit. As $r\to0$,
$$
u_\lambda(r)
=
\left(\frac1{\lambda^2}-\lambda^2\right)r^2
+
\left(\lambda^4-\frac1{\lambda^4}\right)r^4
+
O(r^6).
$$
Therefore
$$
w(r)=Ar^2+Br^4+O(r^6),
$$
where
$$
A=
\left(\frac1{\lambda^2}-\lambda^2\right)
+
k\left(\frac1{\mu^2}-\mu^2\right),
$$
and
$$
B=
\left(\lambda^4-\frac1{\lambda^4}\right)
+
k\left(\mu^4-\frac1{\mu^4}\right).
$$
In the variable $r$,
$$
L(a)
=
\frac{(1+r^2)^2}{4}
\left(
a_{rr}+\frac1r a_r
-\frac{4a(2a-1)(a-1)}{r^2}
\right),
$$
and hence
$$
L(w)=3(A^2+B)r^2+O(r^4).
$$
Moreover,
$$
\partial_\lambda u_\lambda
=
-2\left(\lambda^{-3}+\lambda\right)r^2+O(r^4).
$$
It follows that
$$
\lim_{\theta\to0,\pi}
\frac{L(w(\theta))}{\partial_\lambda u_\lambda(\theta)}
=
-\frac32
\frac{A^2+B}{\lambda^{-3}+\lambda}.
$$
We have $A^2+B=2k\left(\frac1{\mu^2}-\mu^2\right)\lambda^{-2}+O(1)$ as $\lambda\downarrow0$. Therefore
$$
\lim_{\theta\to0,\pi}
\frac{L(w(\theta))}{\partial_\lambda u_\lambda(\theta)}
=
-3k\left(\frac1{\mu^2}-\mu^2\right)\lambda
+O(\lambda^3).
$$
\end{proof}

\begin{corollary}\label{longtimeexist}
Let $a_0$ arise from a smooth equivariant connection and let $a(\theta,t)$ be the corresponding solution to \eqref{thePDE}. Then 
$$
a(\theta,t) \le Ce^{c_{\mathrm{pole}}t} \sin^2\theta \quad \text{for all } \theta\in[0,\pi],\; t\geq0,
$$
where $C,c_{\mathrm{pole}}>0$ are constants depending only on $a_0$. In particular, the solution exists for all time.
\end{corollary}

\begin{proof}
Since $a_0(\theta)=O(\sin^2\theta)$, there exists $C_1>0$ such that
$$
a_0(\theta)\le C_1\sin^2\theta
\qquad\text{for all }\theta\in[0,\pi].
$$
Fix $\mu_1>0$ sufficiently small. Since $u_{\mu_1}(\theta)$ is strictly positive on $(0,\pi)$ and satisfies
$$
u_{\mu_1}(\theta)\sim
\frac{1}{4}\Bigl(\frac1{\mu_1^2}-\mu_1^2\Bigr)\sin^2\theta
\qquad\text{as }\theta\to0,\pi,
$$
there exists $c_1>0$ such that $u_{\mu_1}(\theta)\ge c_1\sin^2\theta$ for all $\theta\in[0,\pi]$. Choose $k>1$ sufficiently large so that $kc_1\ge C_1$.
Now, with this $k$ fixed, decrease $\mu\le\mu_1$ if necessary so that Lemma \ref{Fcomp} applies. Differentiating $u_\lambda$ gives
$$
\partial_\lambda u_\lambda(\theta)
=
-2\lambda\frac{\tan^2(\theta/2)}
       {(\lambda^2+\tan^2(\theta/2))^2}
-
2\lambda\frac{\tan^2(\theta/2)}
       {(1+\lambda^2\tan^2(\theta/2))^2},
$$
hence $\partial_\lambda u_\lambda\le0$ with equality only at $\theta=0,\pi$. Therefore,
$$
ku_\mu\ge ku_{\mu_1}\ge C_1\sin^2\theta\ge a_0.
$$
In particular $a_0(\theta)\le w(\theta)$.

Let $\lambda=\lambda(t)$ with $\lambda(0)=\lambda_0$ and define
$$
w(\theta,t):=u_{\lambda(t)}(\theta)+ku_\mu(\theta).
$$
Then $\partial_t w=\lambda'(t)\partial_\lambda u_{\lambda(t)}$.  Since $\partial_\lambda u_\lambda\le0$, the supersolution condition $\partial_t w\ge L(w)$ is equivalent to
$$
\lambda'(t)\le
\frac{L(w)}{\partial_\lambda u_\lambda}.
$$
By Lemma \ref{Fcomp},
$$
\inf_{\theta\in(0,\pi)}
\frac{L(w)}{\partial_\lambda u_\lambda}
=
-c\lambda+O(\lambda^3)
$$
for some $c>0$. In particular, after possibly decreasing $\lambda_0$, the
right-hand side is strictly negative for all $0<\lambda\le\lambda_0$.
Define $\lambda(t)$ to be the solution of the ODE
$$
\lambda'(t)
=
2\inf_{\theta\in(0,\pi)}
\frac{L(w)}{\partial_\lambda u_\lambda},
\qquad
\lambda(0)=\lambda_0,
$$
so that $w$ is a supersolution. Since $\lambda'(t)<0$, the function
$\lambda(t)$ is decreasing and therefore $0<\lambda(t)\le\lambda_0$ for all $t\geq0$. After possibly decreasing $\lambda_0$ again, there exist
constants $c_2,c_3>0$ such that
$$
-c_3\lambda(t)
\le
\lambda'(t)
\le
-c_2\lambda(t)
$$
for all $t\ge0$. Integrating gives $\lambda_0e^{-c_3t}
\le
\lambda(t)
\le
\lambda_0e^{-c_2t}$. In particular $\lambda(t)$ remains strictly positive for every finite time. The comparison principle yields $a(\theta,t)\le w(\theta,t)$.

Since $u_\lambda$ vanishes quadratically at both poles, there exists a constant
$C=C(\mu,k)>0$ such that
$$
u_\lambda(\theta)+ku_\mu(\theta)
\le
C\Bigl(\lambda^{-2}+1\Bigr)\sin^2\theta
$$
for all $\theta\in[0,\pi]$ and all sufficiently small $\lambda$. Applying the above bound at $\lambda(t)$,
$$
w(\theta,t)
\le
C\Bigl(\lambda(t)^{-2}+1\Bigr)\sin^2\theta.
$$
Using the lower bound $\lambda(t)\ge \lambda_0e^{-c_3t}$, we obtain
$$
w(\theta,t)
\le
C\Bigl(\lambda_0^{-2}e^{2c_3t}+1\Bigr)\sin^2\theta.
$$
After enlarging $C$ if necessary, this yields $w(\theta,t)
\le
Ce^{c_{\mathrm{pole}}t}\sin^2\theta$. Therefore
$$
a(\theta,t)
\le
Ce^{c_{\mathrm{pole}}t}\sin^2\theta
\qquad
\text{for all }\theta\in[0,\pi],\ t\ge0.
$$

For the lower bound, define
$$
\underline w(\theta):=-C_2\sin^2\theta,
$$
where $C_2>0$ will be chosen sufficiently large. A direct computation yields
\begin{align*}
L(\underline w)
&=
2C_2\sin^2\theta-4C_2\cos^2\theta
+
4C_2(2C_2\sin^2\theta+1)(C_2\sin^2\theta+1)\\
&=
6C_2\sin^2\theta
+
12C_2^2\sin^2\theta
+
8C_2^3\sin^4\theta\\
&\ge 0.
\end{align*}
Thus $\underline w$ is a static subsolution. Since $a_0(\theta)=O(\sin^2\theta)$, there exists $C_3>0$ such that
$$
a_0(\theta)\ge -C_3\sin^2\theta
\qquad\text{for all }\theta\in[0,\pi].
$$
Choosing $C_2\ge C_3$ gives $a_0(\theta)\ge\underline w(\theta)$. The comparison principle therefore yields
$$
a(\theta,t)\ge -C_2\sin^2\theta
\qquad\text{for all }t\ge0.
$$

Thus, for every $T>0$,
$$
|a(\theta,t)|
\le
C_T\sin^2\theta
\qquad
\text{for all }
\theta\in[0,\pi],\ t\in[0,T],
$$
where $C_T:=\max\{Ce^{c_{\mathrm{pole}}T},\,C_2\}$. Hence
$$
\Bigl\|
\frac{a(\cdot,t)}{\sin^2\theta}
\Bigr\|_{L^\infty}
\le C_T<\infty
\qquad
\text{for all }t\in[0,T].
$$
Since $T>0$ was arbitrary, the continuation criterion of Lemma~3.1
implies that the solution extends smoothly for all time.
\end{proof}

\section{Asymptotics modulo Metastability}

\subsection{Decaying Supersolutions Foliating $(0,\pi)\times(0,1)$}
For the family $u_\lambda$, we compute $L(u_\lambda)$ explicitly. Since the second-order part of $L$ is linear and $u_\lambda = a^{\mathrm{ASD}} + a^{\mathrm{SD}} - 1$, we have
\begin{align*}
    L(u_\lambda)
&=
L(a^{\mathrm{ASD}}) + L(a^{\mathrm{SD}}) - L(1)
+\frac{1}{\sin^2(\theta)}
\Bigl\{
f(a^{\mathrm{ASD}} + a^{\mathrm{SD}} - 1)
- f(a^{\mathrm{ASD}}) - f(a^{\mathrm{SD}})
\Bigr\}\\
&=\frac{1}{\sin^2(\theta)}
\Bigl\{
f(a^{\mathrm{ASD}} + a^{\mathrm{SD}} - 1)
- f(a^{\mathrm{ASD}}) - f(a^{\mathrm{SD}})
\Bigr\}
\end{align*}
A direct expansion shows that for any $A,B\in\mathbb{R}$,
$$
f(A+B-1)-f(A)-f(B)
=
-24(1-A)(1-B)(A+B-1).
$$
Applying this with $A=a^{\mathrm{ASD}}(\theta)$ and $B=a^{\mathrm{SD}}(\theta)$, and using that $0\leq a^{\mathrm{ASD}},a^{\mathrm{SD}},a^{\mathrm{ASD}}+a^{\mathrm{SD}}-1\leq1$, we obtain
$$
L(u_\lambda)(\theta) \le 0
\qquad \text{for all } \theta\in[0,\pi].
$$
In other words, $u_\lambda$ is a static supersolution of $\eqref{thePDE}$. 

We now let $\lambda$ depend on time. As in Corollary \ref{longtimeexist} the supersolution property is equivalent to
$$
\lambda'(t) \le \frac{L(u_\lambda)(\theta)}{\partial_\lambda u_\lambda(\theta)}
\qquad \text{for all } \theta\in[0,\pi].
$$
As in Lemma \ref{Fcomp} one can show that the infimum of the right hand side is attained at the poles. By symmetry it suffices to consider $\theta\to 0$. Expanding $u_\lambda$ as $\theta\to 0$ gives
$$
u_\lambda(\theta)
=
\frac{1-\lambda^4}{4\lambda^2}\,\theta^2 + O(\theta^4),
$$
which yields
$$
L(u_\lambda)(\theta)
=
\tfrac{3}{2}(\lambda^4-1)\,\theta^2 + O(\theta^4),
\qquad
\partial_\lambda u_\lambda(\theta)
=
-\,\tfrac{\lambda^4+1}{2\lambda^3}\,\theta^2 + O(\theta^4).
$$
Hence
$$
\lim_{\theta\to 0}
\frac{L(u_\lambda)(\theta)}{\partial_\lambda u_\lambda(\theta)}
=
\frac{3\lambda^3(1-\lambda^4)}{\lambda^4+1}
>0
\quad \text{for } 0<\lambda<1.
$$
The same limit holds at $\theta=\pi$. Fix $\lambda_0\in(0,1)$ and let $\lambda(t)$ solve
$$
\lambda'(t)=\tfrac12 \inf_{\theta\in(0,\pi)} \frac{L(u_\lambda)(\theta)}{\partial_\lambda u_\lambda(\theta)}, \qquad \lambda(0)=\lambda_0.
$$
Then $\lambda$ is \textit{increasing} and $u_{\lambda(t)}$ are supersolutions converging to $0$ as $t\to\infty$. 

Moreover, for each fixed $\theta\in(0,\pi)$,
$$
u_\lambda(\theta)\to1
\qquad\text{as }\lambda\downarrow0,
$$
while
$$
u_\lambda(\theta)\sim
\frac{1-\lambda^4}{4\lambda^2}\theta^2
\qquad\text{as }\theta\to0,
$$
and similarly at $\theta=\pi$. Consequently, if
$a(\theta)=O(\sin^2\theta)$ and $\sup a<1$, then
$a\le u_\lambda$ for all sufficiently small $\lambda$.
Indeed, away from the poles the family $u_\lambda$ converges uniformly
to $1$, while near the poles the quadratic coefficient of $u_\lambda$
tends to $+\infty$ as $\lambda\downarrow0$.

\subsection{Eventual Bounding by Perturbed Supersolutions and Convergence to $0$}

In view of the decaying supersolutions constructed in the previous subsection, the main
remaining task is to prove that every solution eventually lies strictly below height $1$.
Once this has been established, those supersolutions provide the required upper bound;
a simple lower barrier will then complete the convergence argument. We first dispose of
the cases in which the initial data are already at or below the threshold. Throughout this subsection, for any solution $a(\theta,t)$ of \eqref{thePDE}, we write
$$
M(t):=\max_{\theta\in[0,\pi]}a(\theta,t).
$$

If $M(0)<1$, no further argument is needed to place the solution below the threshold. If $M(0)=1$, applying
the strong maximum principle to $1-a$ on compact subcylinders of
$(0,\pi)\times(0,\infty)$ shows that $M(t)<1 \qquad \text{for every } t>0$. Now we handle $M(0)>1$.

\begin{lemma}\label{supdecaylem}
Let $a(\theta,t)$ be the solution of \eqref{thePDE} arising from smooth equivariant
initial data, and suppose that $M(0)>1$. Then
$$
M(t)-1\leq \bigl(M(0)-1\bigr)e^{-4t}
\qquad\text{for every }t\geq 0.
$$
\end{lemma}

\begin{proof}
At any time for which $M(t)>1$, the maximum is attained at some
$\theta_t\in(0,\pi)$. At such a point,
$$
a_\theta(\theta_t,t)=0,
\qquad
a_{\theta\theta}(\theta_t,t)\leq 0,
$$
and hence
$$
a_t(\theta_t,t)
\leq
-\frac{4M(t)(2M(t)-1)(M(t)-1)}{\sin^2\theta_t}
\leq -4\bigl(M(t)-1\bigr).
$$
The standard maximum-principle argument therefore gives
$$
M'(t)\leq -4\bigl(M(t)-1\bigr)
$$
for almost every such time.

Let
$$
\nu:=\inf\{t\geq 0:M(t)\leq 1\},
$$
with $\nu=\infty$ if this set is empty. Integrating the preceding inequality on $[0,\nu)$ gives the claimed estimate
there. If $\nu<\infty$, then $M(t)\leq1$ for every $t\geq\nu$, since otherwise
continuity of $M$ would imply that $M$ crosses the level $1$ from below at some
later time, contradicting the differential inequality
$M'(t)\leq-4(M(t)-1)$, which holds whenever $M(t)>1$. The estimate is therefore
automatic after $\nu$.
\end{proof}

In the following Lemma, the parameter $\beta\in(0,1)$ governs the width $\lambda^\beta$ of
the transition region for the cutoff and may be chosen arbitrarily.
Once $\beta$ is fixed, the admissible exponent $\alpha$ and the
required smallness of $\lambda$ depend on the pole-growth constant
$c_{\mathrm{pole}}$, and hence ultimately on the initial data.

\begin{lemma}\label{perturbboundlem}
Let $a(\theta,t)$ be a classical solution of \eqref{thePDE} with
smooth equivariant initial data, and suppose that $M(0)>1$.
Fix $\beta\in(0,1)$ and
$$
0<\alpha<\min\left\{2-2\beta,\frac{8\beta}{c_{\mathrm{pole}}}\right\},
$$
where $c_{\mathrm{pole}}$ is the constant in the pole-growth estimate $a(\theta,t)\leq Ce^{c_{\mathrm{pole}}t}\sin^2\theta$ from Corollary~\ref{longtimeexist}. 

Then, for every sufficiently small $\lambda>0$, there exist a time
$T_\lambda>0$ and a cutoff $\psi_\lambda\in C_c^\infty(0,\pi)$,
with $0\leq\psi_\lambda\leq1$, such that
$$
\psi_\lambda=0
\quad\text{on}\quad
[0,\lambda^\beta]\cup[\pi-\lambda^\beta,\pi],
\qquad
\psi_\lambda=1
\quad\text{on}\quad
[2\lambda^\beta,\pi-2\lambda^\beta],
$$
and
$$
a(\cdot,T_\lambda)
\leq
u_\lambda+\lambda^\alpha\psi_\lambda
\qquad\text{on }[0,\pi].
$$
\end{lemma}
\begin{proof}
By Lemma~\ref{supdecaylem},
$$
M(t)-1
\leq
\bigl(M(0)-1\bigr)e^{-4t}
\qquad
(t\geq0).
$$
Thus, for sufficiently small $\lambda$, we may define
$$
T_\lambda
=
\frac14
\log\left(\frac{2(M(0)-1)}{\lambda^\alpha}\right)>0,
$$
so that $M(T_\lambda)\leq1+\frac12\lambda^\alpha$. We first consider the region $2\lambda^\beta\leq\theta\leq\pi-2\lambda^\beta$ where $\psi_\lambda=1$. On this region there exists $C,c>0$ such that $c\lambda^\beta\leq r:=\tan(\theta/2)\leq C\lambda^{-\beta}$. The explicit
formula for $u_\lambda$ gives
$$
1-u_\lambda
=
\frac{\lambda^2}{\lambda^2+r^2}
+
\frac{\lambda^2r^2}{1+\lambda^2r^2}
\leq
\frac{\lambda^2}{r^2}+\lambda^2r^2
\leq
C\lambda^{2-2\beta}.
$$
Since $\alpha<2-2\beta$, after decreasing $\lambda$ we have $1-u_\lambda\leq\frac12\lambda^\alpha$. Consequently,
$$
u_\lambda+\lambda^\alpha\psi_\lambda
\geq
1+\frac12\lambda^\alpha
\geq
M(T_\lambda).
$$
It remains to consider the two poles. By reflection it is
enough to treat $0\leq\theta\leq2\lambda^\beta$. On this interval
$r\leq C\lambda^\beta$, and
$$
u_\lambda(\theta)
=
\frac{(1-\lambda^4)r^2}
     {(\lambda^2+r^2)(1+\lambda^2r^2)}.
$$
Since $\lambda^2r^2=o(1)$, it follows that
$$
u_\lambda(\theta)
\geq
c\frac{r^2}{\lambda^2+r^2}
\geq
c\min\left\{\frac{r^2}{\lambda^2},1\right\}.
$$
On the other hand, the pole-growth estimate and the definition of
$T_\lambda$ give
$$
a(\theta,T_\lambda)
\leq
C\lambda^{-c_{\mathrm{pole}}\alpha/4}\sin^2\theta
\leq
C\lambda^{-c_{\mathrm{pole}}\alpha/4}r^2.
$$
If $0\leq r\leq\lambda$, then $u_\lambda\geq c\frac{r^2}{\lambda^2}$, and hence $a(\theta,T_\lambda)\leq u_\lambda(\theta)$ for small
$\lambda$, because $\frac{c_{\mathrm{pole}}\alpha}{4}<2\beta<2$. If instead $\lambda\leq r\leq C\lambda^\beta$, then
$u_\lambda\geq c$, whereas
$$
a(\theta,T_\lambda)
\leq
C\lambda^{2\beta-c_{\mathrm{pole}}\alpha/4}
=o(1)
$$
by the assumption $\alpha<8\beta/c_{\mathrm{pole}}$. 
\end{proof}

\begin{remark}\label{perturbationcutoff}
For an example of the cutoffs in Lemma~\ref{perturbboundlem}, let
$\chi\in C^\infty(\mathbb R)$ be nondecreasing and satisfy
$$
0\leq\chi\leq1,
\qquad
\chi(x)=0\ \text{for }x\leq1,
\qquad
\chi(x)=1\ \text{for }x\geq2,
$$
and define
$$
\psi_\lambda(\theta)
=
\chi\!\left(\frac{\theta}{\lambda^\beta}\right)
\chi\!\left(\frac{\pi-\theta}{\lambda^\beta}\right).
$$
Then $\psi_\lambda$ is symmetric and satisfies
$$
\int_0^\pi
\left(
|\partial_\theta\psi_\lambda|^2\sin\theta
+
\frac{\psi_\lambda^2}{\sin\theta}
\right)\,d\theta
\leq
C|\log\lambda|.
$$
Indeed, $\partial_\theta\psi_\lambda$ is supported in two collars of
length $O(\lambda^\beta)$, where
$$
|\partial_\theta\psi_\lambda|
\leq
C\lambda^{-\beta},
\qquad
\sin\theta\leq C\lambda^\beta,
$$
while
$$
\int_{\lambda^\beta}^{\pi-\lambda^\beta}
\frac{d\theta}{\sin\theta}
\leq
C|\log\lambda|.
$$
\end{remark}

With Lemma \ref{perturbboundlem} in hand, conditional on the results of the next section, we give the proof of the main result as stated in Theorem \ref{1.1}. 

\begin{proof}[Proof of Theorem \ref{1.1}]
By the reduction at the beginning of this subsection, for the upper bound it remains only to consider the
case $M(0)>1$. By Lemma \ref{perturbboundlem}, for every sufficiently small
$\lambda>0$ there exists $T_\lambda>0$ such that
$$
a(\cdot,T_\lambda)\leq u_\lambda+\lambda^\alpha\psi_\lambda.
$$
Let $\bar a$ denote the solution with initial data $u_\lambda+\lambda^\alpha\psi_\lambda$. For sufficiently small $\lambda$, Theorem \ref{eventualdropthm} gives a time
$t_*>0$ such that $\sup_{\theta\in[0,\pi]}\bar a(\theta,t_*)<1$. The comparison principle therefore yields $M(T_\lambda+t_*)<1$. Section~4.1 now gives a function $l(t)$, increasing to $1$ as $t\to\infty$, such that $a(\theta,t)\leq u_{\ell(t)}(\theta)$ for $t\geq t_0:=T_\lambda+t_*$.

In order to get smooth convergence, we now work with the smooth equivariant variable $b(\theta,t):=\frac{a(\theta,t)}{\sin^2\theta}$. One has
$$
\frac{u_\lambda(\theta)}{\sin^2\theta}
=
\frac{(1-\lambda^4)(1+r^2)^2}
{4(\lambda^2+r^2)(1+\lambda^2r^2)}.
$$
Since
$$
(\lambda^2+r^2)(1+\lambda^2r^2)
-
\lambda^2(1+r^2)^2
=
r^2(1-\lambda^2)^2\geq0,
$$
it follows that
$$
b(\theta,t)
\leq
\frac{1-\ell(t)^4}{4\ell(t)^2}.
$$
Hence 
$\limsup_{t\to\infty}
\sup_{\theta\in[0,\pi]}b(\theta,t)
\leq0$. 

Since $b(\cdot,t_0)$ is smooth on $S^6$, there is $C>0$ such
that
$$
a(\theta,t_0)\geq-C\sin^2\theta.
$$
Define
$$
\underline a(\theta,t)
:=
-Ce^{-6(t-t_0)}\sin^2\theta.
$$
A direct computation gives
$$
L(-h\sin^2\theta)
=
6h\sin^2\theta
+
12h^2\sin^2\theta
+
8h^3\sin^4\theta
\geq
6h\sin^2\theta
\qquad
(h\geq0).
$$
Taking $h(t)=Ce^{-6(t-t_0)}$, we have
$$
\partial_t\underline a
=
6Ce^{-6(t-t_0)}\sin^2\theta
\leq
L(\underline a).
$$
Thus $-Ce^{-6(t-t_0)}$ is a subsolution of the regularised equation \eqref{beqn} on $\mathbb S^6$ and, by comparison, $b(\theta,t)
\geq
-Ce^{-6(t-t_0)}$. Combining the upper and lower bounds gives
$$
\|b(\cdot,t)\|_{L^\infty(\mathbb S^6)}
\longrightarrow0.
$$
Since $b$ satisfies a smooth semilinear parabolic equation on the compact manifold $\mathbb S^6$, standard parabolic regularity upgrades $\|b(\cdot,t)\|_{L^\infty(\mathbb S^6)}\to0$ to $b(\cdot,t)\to0$ smoothly on $\mathbb S^6$. Since $a(\theta,t)=\sin^2\theta\,b(\theta,t)$, it follows that $a(\cdot,t)\to0$ smoothly on $[0,\pi]$.
\end{proof}

\section{Metastability}

The goal of this section is to prove that, for sufficiently small $\lambda_0$, the flow started from
$$
u_{\lambda_0}+\lambda_0^\alpha\psi_{\lambda_0}
$$
eventually satisfies
$$
\sup_{\theta\in[0,\pi]}a(\theta,t)<1.
$$

Throughout this section we work with the weighted inner product
$$
\langle f,g\rangle_w
:=
\int_0^\pi f(\theta)g(\theta)\sin\theta\,d\theta,
\qquad
L_w^p
:=
L^p\bigl((0,\pi),\sin\theta\,d\theta\bigr),
$$
and at times use the coordinate
$$
z=z(\theta):=\log\tan(\theta/2).
$$
When a scale $\lambda$ is fixed, we write
$$
s:=-\log\lambda.
$$
In these coordinates we have the explicit formula
\begin{equation}\label{cylindricformula}
u_\lambda(z)
=
Q(z+s)+Q(s-z)-1,
\qquad
Q(y):=\frac{1}{1+e^{-2y}}.
\end{equation}

As above, $C>0$ may change from line to line and is independent of sufficiently small $\lambda$. We write $A\asymp B$ to denote two-sided comparability.

In \S5.1 we introduce a localized tangent vector $\eta_\lambda$ and modulation coordinates, proving that every solution sufficiently close to the metastable family $\{u_\lambda\}$ admits a canonical decomposition
$$
a=u_{\lambda(t)}+v,
\qquad
\langle v,\eta_{\lambda(t)}\rangle_w=0.
$$
The problem is thereby reduced to controlling the modulation parameter $\lambda(t)$ and the error $v(t)$. In \S5.2 we establish the required coercivity estimates for the linearised operator, while in \S5.3 we prove decay of the error together with logarithmic control of the modulation parameter, ultimately showing that the solution passes below height $1$.

The analysis in this section is adapted from the metastability framework
of Carr--Pego \cite{carr1989metastable}.  The main proof ingredients follow
the strategy developed there, making the necessary adaptations for this more singular setting.

\subsection{Tubular Coordinates}

The initial condition $u_{\lambda_0}+\lambda_0^\alpha\psi_{\lambda_0}$ is symmetric under the reflection $\theta\mapsto\pi-\theta$. The following observation shows that this symmetry is preserved by the flow. This will be convenient throughout the remainder of the section, particularly in the spectral analysis.

\begin{lemma}
Let $a(\theta,t)$ be the smooth solution of the reduced Yang--Mills flow with initial data $a_0(\theta)$. If $a_0$ is symmetric about $\theta=\pi/2$ then $a(\cdot,t)$ is symmetric about $\theta=\pi/2$ for all $t\geq0$.
\end{lemma}

\begin{proof}
Set
$$
\widetilde a(\theta,t):=a(\pi-\theta,t).
$$
Since
$$
\partial_t\widetilde a(\theta,t)=a_t(\pi-\theta,t),
\qquad
\partial_\theta\widetilde a(\theta,t)=-a_\theta(\pi-\theta,t),
$$
$$
\partial_{\theta\theta}\widetilde a(\theta,t)
=
a_{\theta\theta}(\pi-\theta,t),
$$
and
$$
\sin(\pi-\theta)=\sin\theta,
\qquad
\cot(\pi-\theta)=-\cot\theta,
$$
the function $\widetilde a$ satisfies the same equation as $a$. Moreover,
$$
\widetilde a(\theta,0)
=
a_0(\pi-\theta)
=
a_0(\theta).
$$
The result follows from uniqueness established in Lemma \ref{shrttime} and Corollary \ref{longtimeexist}.
\end{proof}

We use the notation
$$
\tau_\lambda:=\partial_\lambda u_\lambda
$$
and
$$
\mathcal A_\lambda
=
-\partial_{\theta\theta}
-\cot\theta\,\partial_\theta
+
\frac{4\bigl(6u_\lambda(\theta)^2-6u_\lambda(\theta)+1\bigr)}
{\sin^2\theta}.
$$
Here $-\mathcal A_\lambda$ is the linearisation of $L$ around $u_\lambda$. We now define the localised tangent. Fix an even cutoff function $\zeta\in C_c^\infty(\mathbb R)$ such that
$$
0\leq\zeta\leq1,
\qquad
\zeta(z)=1\quad\text{for }|z|\leq R_\zeta,
\qquad
\zeta(z)=0\quad\text{for }|z|\geq2R_\zeta,
$$
where $R_\zeta>0$ is fixed. Define
$$
\zeta_\lambda(z)
:=
\zeta(z+s)
+
\zeta(s-z)
$$
and
$$
\eta_\lambda(\theta(z))
:=
\lambda^{-2}\zeta_\lambda(z)\tau_\lambda(\theta(z)).
$$
Thus $\eta_\lambda$ is supported in fixed-width $z$-neighbourhoods of the two transition layers at $\pm s$. 

The asymptotics as $\lambda\downarrow0$ of these objects will be crucial in the subsequent proofs. They are recorded in the following proposition and are all shown by direct computation from the explicit formulae for
$u_\lambda$, $\tau_\lambda$, and $\eta_\lambda$. We omit the proofs as they are not illuminating.

\begin{proposition}\label{norms}
For all sufficiently small $\lambda$,
$$
\|\tau_\lambda\|_{L_w^2}^2
=
\frac{16}{3}
+
O\left(
\lambda^2\log\frac1\lambda
\right),
\qquad
\|\tau_\lambda\|_{L^\infty([0,\pi])}
\leq
\frac{C}{\lambda},
$$
$$
|\tau_\lambda(\theta)|\sin^2\theta
\leq
C\lambda,
$$
$$
\langle\tau_\lambda,\eta_\lambda\rangle_w
\asymp
\lambda^{-2},
\qquad
\|\eta_\lambda\|_{L_w^1}
\leq
C\lambda^{-1},
\qquad
\|\partial_\lambda\eta_\lambda\|_{L_w^1}
\leq
C\lambda^{-2},
$$
$$
\|L(u_\lambda)\|_{L_w^2}
\leq
C\lambda^3,
\qquad
\|\mathcal A_\lambda\tau_\lambda\|_{L_w^2}
\leq
C\lambda^2.
$$
\end{proposition}
\noindent
The estimate
$$
\|\mathcal A_\lambda\tau_\lambda\|_{L_w^2}
\leq
C\lambda^2
$$
shows that $\tau_\lambda$ is an approximate kernel element for the linearised operator when $\lambda$ is small. This can be seen as a key source of the metastable behaviour.

\begin{lemma}\label{tubeee_eta}
Fix $\delta>0$. There exists $\lambda_*>0$ such that the following holds. Define
$$
\mathcal U_{\lambda_*,\delta}
:=
\left\{
a\in L^\infty([0,\pi]):
\ \|a-u_\mu\|_{L^\infty}<\mu^\delta
\text{ for some }\mu\in(0,\lambda_*/2]
\right\}.
$$
For every $a\in\mathcal U_{\lambda_*,\delta}$ with say $\|a-u_{\mu_0}\|_{L^\infty}<\mu_0^\delta$, there is a unique zero $\Lambda(a)\in(0,\lambda_*]$ of
$\Gamma(a,\cdot):=
\langle a-u_\lambda,\eta_\lambda\rangle_w$ satisfying
\begin{equation}\label{witnessbound}
|\Lambda(a)-\mu_0|
\leq
C\mu_0^{1+\delta}.
\end{equation}
The map $a\mapsto\Lambda(a)$ is locally $C^1$.

Let $a(\theta,t)$ be a solution of \eqref{thePDE}. On the maximal
interval $[0,T)$ for which
$$
a(\cdot,t)\in\mathcal U_{\lambda_*,\delta},
$$
set
$$
\lambda(t):=\Lambda(a(\cdot,t)),
\qquad
v(\theta,t):=a(\theta,t)-u_{\lambda(t)}(\theta).
$$
Then $\lambda$ is continuous on $[0,T)$, is $C^1$ on
$(0,T)$, and
$$
\langle v(t),\eta_{\lambda(t)}\rangle_w=0,
\qquad
\|v(\cdot,t)\|_{L^\infty}
\leq
C\lambda(t)^\delta.
$$
\end{lemma}

\begin{proof}
Fix $a\in\mathcal U_{\lambda_*,\delta}$, and choose an admissible
scale $\mu_0\in(0,\lambda_*/2]$, meaning that $\|a-u_{\mu_0}\|_{L^\infty}<\mu_0^\delta$. Let $K\geq1$ be a constant, independent of $a$ and
$\mu_0$, to be fixed below, and consider
$$
I_{\mu_0}
:=
\left[
\mu_0-K\mu_0^{1+\delta},
\mu_0+K\mu_0^{1+\delta}
\right].
$$
After decreasing $\lambda_*$, every $\mu\in I_{\mu_0}$
satisfies $\frac12\mu_0\leq\mu\leq2\mu_0$. Proposition~\ref{norms} and the mean value theorem therefore give
$$
\|u_\mu-u_{\mu_0}\|_{L^\infty}
\leq
|\mu-\mu_0|
\sup_{\xi\in I_{\mu_0}}\|\tau_\xi\|_{L^\infty}
\leq
C K\mu_0^\delta.
$$
Consequently,
\begin{equation}\label{aclose}
\|a-u_\mu\|_{L^\infty}
\leq
K'\mu_0^\delta,
\end{equation}
where $K'$ depends only on $K$ and the constants in
Proposition~\ref{norms}.

Differentiating $\Gamma$ in $\mu$ gives
$$
\partial_\mu \Gamma(a,\mu)
=
-\langle\tau_\mu,\eta_\mu\rangle_w
+
\langle a-u_\mu,\partial_\mu\eta_\mu\rangle_w.
$$
By Proposition \ref{norms}, there are fixed constants
$c_1,C_1>0$ such that
$$
\langle\tau_\mu,\eta_\mu\rangle_w\geq c_1\mu^{-2},
\qquad
\|\partial_\mu\eta_\mu\|_{L_w^1}\leq C_1\mu^{-2}.
$$
Using \eqref{aclose} and $\mu\asymp\mu_0$,
we obtain
$$
\partial_\mu \Gamma(a,\mu)
\leq
-c_1\mu^{-2}
+
C_1K'\mu_0^\delta\mu^{-2}.
$$
Once $K$ has been fixed, we may decrease $\lambda_*$, depending
only on $\delta$ and the constants above, so that
\begin{equation}\label{moddecreaseunifrm}
\partial_\mu \Gamma(a,\mu)
\leq
-c_2\mu_0^{-2}
\end{equation}
for some fixed $c_2>0$.

At the admissible scale $\mu_0$,
$$
|\Gamma(a,\mu_0)|
\leq
\|a-u_{\mu_0}\|_{L^\infty}
\|\eta_{\mu_0}\|_{L_w^1}
\leq
C_2\mu_0^{\delta-1},
$$
where $C_2$ is independent of $a$ and $\mu_0$. Choose
$K$ so large that $Kc_2>2C_2$. By \eqref{moddecreaseunifrm},
$$
\begin{aligned}
\Gamma\bigl(a,\mu_0-K\mu_0^{1+\delta}\bigr)
&\geq
\Gamma(a,\mu_0)+Kc_2\mu_0^{\delta-1}>0,\\
\Gamma\bigl(a,\mu_0+K\mu_0^{1+\delta}\bigr)
&\leq
\Gamma(a,\mu_0)-Kc_2\mu_0^{\delta-1}<0.
\end{aligned}
$$
The intermediate value theorem gives a zero $\Lambda_{\mu_0}(a)\in I_{\mu_0}$. It is unique in this interval by
\eqref{moddecreaseunifrm}, and
\begin{equation}\label{localrootbnd}
|\Lambda_{\mu_0}(a)-\mu_0|
\leq
K\mu_0^{1+\delta}.
\end{equation}

We next show that the zero does not depend on the admissible scale
used to construct it. The formula \eqref{cylindricformula} implies that, for sufficiently small $\lambda,\mu$,
\begin{equation}\label{profiletubeeq}
\|u_\lambda-u_\mu\|_{L^\infty}
\geq
c\min\left\{
1,\left|\log\frac{\mu}{\lambda}\right|
\right\}.
\end{equation}
Indeed, by symmetry in $\lambda$ and $\mu$ we may assume $\mu\leq\lambda$. Evaluating at
$z=\log\lambda$ gives
$$
u_\mu(\log\lambda)-u_\lambda(\log\lambda)
\geq
Q\left(\log\frac{\lambda}{\mu}\right)-Q(0),
$$
from which \eqref{profiletubeeq} follows.

Suppose that $\mu_0$ and $\mu_1$ are both admissible
scales for $a$. Then
$$
\|u_{\mu_0}-u_{\mu_1}\|_{L^\infty}
\leq
\mu_0^\delta+\mu_1^\delta.
$$
After decreasing $\lambda_*$, the right-hand side is smaller than
the fixed constant in the first alternative of
\eqref{profiletubeeq}. Hence
$$
\left|\log\frac{\mu_1}{\mu_0}\right|
\leq
C(\mu_0^\delta+\mu_1^\delta).
$$
It follows first that $\mu_1\asymp\mu_0$, and then that
\begin{equation}\label{scalesclose}
|\mu_1-\mu_0|
\leq
C\mu_0^{1+\delta}.
\end{equation}

By \eqref{localrootbnd} and
\eqref{scalesclose}, both locally constructed zeros
lie in an interval
$$
J_{\mu_0}
=
\left[
\mu_0-C\mu_0^{1+\delta},
\mu_0+C\mu_0^{1+\delta}
\right],
$$
after increasing $C$ if necessary. For
$\mu\in J_{\mu_0}$, one still has
$\mu\asymp\mu_0$ and $\|a-u_\mu\|_{L^\infty}
\leq
C\mu_0^\delta$. The preceding calculation therefore gives, after one final decrease
of $\lambda_*$,
$$
\partial_\mu \Gamma(a,\mu)
\leq
-c\mu_0^{-2}
\qquad
(\mu\in J_{\mu_0}).
$$
Thus $\Gamma(a,\cdot)$ is strictly decreasing on
$J_{\mu_0}$, and the two zeros agree.

We denote their common value by $\Lambda(a)$. Estimate
\eqref{witnessbound} follows from
\eqref{localrootbnd}, and decreasing $\lambda_*$ ensures
that
$$
0<\Lambda(a)\leq\lambda_*.
$$
Since
$$
\partial_\mu \Gamma(a,\Lambda(a))\neq0,
$$
the implicit function theorem shows that $a\mapsto\Lambda(a)$ is
locally $C^1$.

We now apply the construction along the flow. Set
$$
\lambda(t):=\Lambda(a(\cdot,t)),
\qquad
v(\theta,t):=a(\theta,t)-u_{\lambda(t)}(\theta).
$$
By definition $\langle v(t),\eta_{\lambda(t)}\rangle_w=0$. The continuity of $a(\cdot,t)$ in $L^\infty$, together with the
local $C^1$-dependence of $\Lambda$, gives continuity of
$\lambda$. Since the solution is smooth for $t>0$, the implicit
function theorem also gives $\lambda\in C^1((0,T))$.

Finally, for each $t<T$, choose an admissible scale
$\mu_t\in(0,\lambda_*/2]$. Then
$$
\|a(\cdot,t)-u_{\mu_t}\|_{L^\infty}<\mu_t^\delta
$$
and
$$
|\lambda(t)-\mu_t|
\leq
C\mu_t^{1+\delta}.
$$
In particular, $\lambda(t)\asymp\mu_t$. Proposition~\ref{norms} and
the mean value theorem give
$$
\|u_{\lambda(t)}-u_{\mu_t}\|_{L^\infty}
\leq
C\frac{|\lambda(t)-\mu_t|}{\mu_t}
\leq
C\mu_t^\delta.
$$
Therefore
$$
\begin{aligned}
\|v(\cdot,t)\|_{L^\infty}
&\leq
\|a(\cdot,t)-u_{\mu_t}\|_{L^\infty}
+
\|u_{\mu_t}-u_{\lambda(t)}\|_{L^\infty}\\
&\leq
C\mu_t^\delta
\leq
C\lambda(t)^\delta.
\end{aligned}
$$
\end{proof}

\begin{remark}\label{initialscalerem}
For the initial data $a_0
=
u_{\lambda_0}
+
\lambda_0^\alpha\psi_{\lambda_0}$, the support of $\eta_{\lambda_0}$ is contained in regions where $\theta\lesssim\lambda_0$ or $\pi-\theta\lesssim\lambda_0$. Since $\beta<1$, these regions lie inside the set where
$\psi_{\lambda_0}=0$, for sufficiently small $\lambda_0$.
Hence
$$
\langle a_0-u_{\lambda_0},\eta_{\lambda_0}\rangle_w=0.
$$
By the uniqueness in Lemma~\ref{tubeee_eta},
$$
\lambda(0)=\lambda_0,
\qquad
v(\theta,0)=\lambda_0^\alpha\psi_{\lambda_0}(\theta).
$$
\end{remark}

\subsection{Spectral Analysis}
We first give the functional-analytic formulation of
$\mathcal A_\lambda$. In the coordinate $z=\log\tan\frac{\theta}{2}$ we have
$$
\sin\theta=\sech z,
\qquad
d\theta=\sech z\,dz,
$$
and we will write
$$
V(z):=v(\theta(z)).
$$
Then
\begin{equation}
\|v\|_{L_w^2}^2
=
\int_{\mathbb R}|V(z)|^2\sech^2z\,dz.
\end{equation}
Thus $L_w^2$ is identified with $L^2(\mathbb R,\sech^2z\,dz)$. Writing
$$
W_\lambda(z)
:=
4\bigl(6u_\lambda(\theta(z))^2-6u_\lambda(\theta(z))+1\bigr),
\qquad
\mathscr L_\lambda:=-\partial_z^2+W_\lambda,
$$
we have
\begin{equation}
\mathcal A_\lambda v
=
\cosh^2z\,\mathscr L_\lambda V.
\end{equation}
Thus the eigenvalue equation
$\mathcal A_\lambda v=\mu v$ becomes
$$
-V''+W_\lambda V=\mu\sech^2z\,V.
$$
The quadratic form associated with $\mathcal A_\lambda$ will be denoted by
\begin{equation}
\mathcal E_\lambda(v)
:=
\langle\mathcal A_\lambda v,v\rangle_w
=
\int_{\mathbb R}
\left(
|V'(z)|^2+W_\lambda(z)|V(z)|^2
\right)\,dz.
\end{equation}
When working in $z$-coordinates we use the same notation $\mathcal E_\lambda(V)$ for the right-hand side.
We write
$$
\mathcal H
:=
\left\{
v:
\int_0^\pi
\left(
|v_\theta|^2\sin\theta
+
\frac{v^2}{\sin\theta}
\right)d\theta<\infty
\right\};
$$
this is motivated by 
\begin{equation}\label{H1identity}
\|v\|_{\mathcal H}^2
=
\int_{\mathbb R}
\left(
|V'|^2+V^2
\right)dz
=
\|V\|_{H^1(\mathbb R)}^2.
\end{equation}

\begin{lemma}
For each $\lambda>0$, the form $\mathcal E_\lambda$, with form domain $\mathcal H$, is densely
defined, closed, and bounded from below on $L_w^2$. Consequently,
there exists a unique self-adjoint operator, denoted again by
$\mathcal A_\lambda$, associated with $\mathcal E_\lambda$. In $z$-coordinates, on $C_c^\infty(\mathbb R)$ it is given by $V\mapsto\cosh^2z\,\mathscr L_\lambda V$.
\end{lemma}

\begin{proof}
Since $W_\lambda(z)\to4$ as $|z|\to\infty$, there exists a compact
interval $K_\lambda\subset\mathbb R$ such that $W_\lambda(z)\ge2$ for $z\notin K_\lambda$. On $K_\lambda$, the weight $\sech^2 z$ is bounded below by a positive
constant. Hence there exist constants $c_1(\lambda),C_1(\lambda)>0$ such
that
$$
\mathcal E_\lambda(V)
+
C_1(\lambda)\|v\|_{L_w^2}^2
\ge
c_1(\lambda)\|V\|_{H^1(\mathbb R)}^2.
$$
Since $W_\lambda\in L^\infty(\mathbb R)$ and $\sech^2z\leq1$,
$$
\mathcal E_\lambda(V)
+
C_1(\lambda)\|v\|_{L_w^2}^2
\le
C_2(\lambda)
\|V\|_{H^1(\mathbb R)}^2
$$
after increasing $C_1(\lambda)$. Thus the shifted form norm is
equivalent to the $H^1$-norm. In particular,
$\mathcal E_\lambda$ is closed and bounded from below.
The representation theorem therefore yields a unique self-adjoint
operator associated with $\mathcal E_\lambda$. An integration by
parts shows that, on $C_c^\infty(\mathbb R)$,
$$
(\mathcal A_\lambda v)(\theta(z))
=
\cosh^2z\,\mathscr L_\lambda V.
$$
\end{proof}

We now study the two transition layers. Recall $Q$ from \eqref{cylindricformula} and set
$$
q(y):=Q'(y)=\frac12\sech^2y.
$$
Then
$$
u_\lambda(\theta(z))
=
Q(z+s)+Q(s-z)-1.
$$
Define
$$
W_*(y)
:=
4\bigl(6Q(y)^2-6Q(y)+1\bigr)
=
4-6\sech^2y
$$
and
$$
\mathscr L_*:=-\partial_y^2+W_*(y).
$$
Differentiating the equation for $Q$ gives
$$
\mathscr L_*q=0,
\qquad
q>0.
$$
\begin{lemma}
\label{singlefrontlem}
There exists $c>0$ such that, for every $\phi\in H^1(\mathbb R)$
satisfying
$$
\int_{\mathbb R}\phi(y)q(y)\,dy=0,
$$
one has
$$
\int_{\mathbb R}
\left(
|\phi'(y)|^2+W_*(y)\phi(y)^2
\right)dy
\geq
c\|\phi\|_{H^1(\mathbb R)}^2.
$$
\end{lemma}

\begin{proof}
Since $W_*(y)=4-6\sech^2 y$, the operator admits the factorisation
$$
\mathscr L_*
=
\left(-\partial_y+2\tanh y\right)
\left(\partial_y+2\tanh y\right).
$$
Consequently, for every $\phi\in H^1(\mathbb R)$,
$$
\int_{\mathbb R}
\left(
|\phi'|^2+W_*\phi^2
\right)dy
=
\int_{\mathbb R}
\left|\phi'+2\tanh(y)\phi\right|^2dy
\geq 0.
$$
Equality holds if and only if $\phi'+2\tanh(y)\phi=0$ in the distributional sense, and hence
$$
\phi=C\sech^2 y=2Cq.
$$
Thus $\ker\mathscr L_*=\operatorname{span}\{q\}$. 

Since multiplication by $W_*-4=-6\sech^2y$ is relatively compact, Weyl's theorem gives
$$
\sigma_{\mathrm{ess}}(\mathscr L_*)=\sigma_{\mathrm{ess}}(-\partial_y^2+4)=[4,\infty).
$$
The spectrum below $4$ is therefore discrete. Since $0$ is a simple
eigenvalue, it is isolated, and the spectral theorem yields $c_1>0$
such that
$$
\int_{\mathbb R}
\left(
|\phi'|^2+W_*\phi^2
\right)dy
\geq
c_1\|\phi\|_{L^2(\mathbb R)}^2
$$
whenever $\phi\perp q$.

Finally, $W_*\geq-2$, so
$$
\|\phi'\|_{L^2}^2
\leq
\int_{\mathbb R}
\left(
|\phi'|^2+W_*\phi^2
\right)dy
+
2\|\phi\|_{L^2}^2.
$$
Combining the preceding two estimates gives
$$
\|\phi\|_{H^1(\mathbb R)}^2
\leq
\left(1+\frac{3}{c_1}\right)
\int_{\mathbb R}
\left(
|\phi'|^2+W_*\phi^2
\right)dy,
$$
which proves the result.
\end{proof}

Define the normalized symmetric two-front tangent
\begin{equation}
\widehat{\tau}_\lambda(z)
:=
q(z+s)+q(s-z).
\end{equation}
The explicit formula for $\tau_\lambda$ gives
\begin{equation}\label{twofrontrelation}
\tau_\lambda(\theta(z))
=
-\lambda^{-1}\widehat{\tau}_\lambda(z).
\end{equation}
Thus
\begin{equation}\label{etatwofrontrelation}
\eta_\lambda(\theta(z))
=
-\lambda^{-3}\zeta_\lambda(z)\widehat{\tau}_\lambda(z).
\end{equation}

\begin{lemma}
\label{twofrontlem}
For all sufficiently small $\lambda$,
$$
\|\widehat{\tau}_\lambda\|_{H^1(\mathbb R)}
\leq
C.
$$
Moreover,
\begin{equation}\label{weightedest}
c\lambda^2
\leq
\int_{\mathbb R}
\zeta_\lambda \widehat{\tau}_\lambda^2\sech^2z\,dz
\leq
C\lambda^2
\end{equation}
and
\begin{equation}\label{L1est}
\int_{\mathbb R}
\zeta_\lambda|\widehat{\tau}_\lambda|\sech^2z\,dz
\leq
C\lambda^2.
\end{equation}
Finally,
\begin{equation}\label{unweightedest}
\|\mathscr L_\lambda \widehat{\tau}_\lambda\|_{L^2(\mathbb R)}
\longrightarrow0.
\end{equation}
\end{lemma}

\begin{proof}
Since translation and reflection preserve the $H^1(\mathbb R)$-norm $\|\widehat{\tau}_\lambda\|_{H^1}
\le
2\|q\|_{H^1}\leq C$. Near the left front, write $y=z+s$. By symmetry, and since the two supports comprising $\zeta_\lambda$ are disjoint for small $\lambda$,
$$
\lambda^{-2}\int_{\mathbb R}\zeta_\lambda \widehat{\tau}_\lambda^2\sech^2z\,dz
=
2\int_{\mathbb R}\zeta(y)
\bigl(q(y)+q(2s-y)\bigr)^2
\lambda^{-2}\sech^2(y-s)\,dy.
$$
Here
$$
\lambda^{-2}\sech^2(y-s)
=\frac{4e^{2y}}{(1+\lambda^2e^{2y})^2}
\to4e^{2y}
$$
locally uniformly, while $q(2s-y)\to0$ uniformly on $\operatorname{supp}\zeta$. Since the integral is over the fixed compact set $\operatorname{supp}\zeta$, dominated convergence gives
$$
\lambda^{-2}\int_{\mathbb R}\zeta_\lambda \widehat{\tau}_\lambda^2\sech^2z\,dz
\to
8\int_{\mathbb R}\zeta(y)q(y)^2e^{2y}\,dy>0.
$$
This proves \eqref{weightedest}. The proof of
\eqref{L1est} is entirely analogous, replacing
$q^2$ by $|q|$.

Finally,
$$
\begin{aligned}
\mathscr L_\lambda q(z+s)
&=
\bigl(W_\lambda(z)-W_*(z+s)\bigr)
q(z+s),\\
\mathscr L_\lambda q(s-z)
&=
\bigl(W_\lambda(z)-W_*(s-z)\bigr)
q(s-z).
\end{aligned}
$$
After translating to either front, the coefficient in parentheses
converges locally uniformly to zero and remains uniformly bounded.
Dominated convergence therefore gives
\eqref{unweightedest}.
\end{proof}
\noindent
The following is the principal compactness estimate.

\begin{lemma}
\label{twofrontunweightedlem}
There exists $c>0$ such that, for every sufficiently small
$\lambda>0$ and every even
$\varphi\in H^1(\mathbb R)$ satisfying
$$
\int_{\mathbb R}\varphi(z)\widehat{\tau}_\lambda(z)\,dz=0,
$$
one has
\begin{equation}
\mathcal E_\lambda(\varphi)
\geq
c\|\varphi\|_{H^1(\mathbb R)}^2.
\end{equation}
\end{lemma}
\begin{proof}
Suppose the result were false. Then there exist
$\lambda_n\downarrow0$ and even functions $\varphi_n\in H^1(\mathbb R)$
such that
$$
\|\varphi_n\|_{H^1}=1,
\qquad
\int_{\mathbb R}\varphi_n\widehat{\tau}_{\lambda_n}\,dz=0,
\qquad
\limsup_{n\to\infty}\mathcal E_{\lambda_n}(\varphi_n)\leq0.
$$
Write $s_n:=-\log\lambda_n$. Since $W_{\lambda_n}\simeq 4$ away from the transition layers, for every sufficiently large $R$, one has
\begin{equation}\label{positiveawayfronts}
W_{\lambda_n}(z)\geq3
\end{equation}
whenever $\operatorname{dist}(z,\{-s_n,s_n\})\geq R$ and $n$ is sufficiently large.
Choose a smooth quadratic partition of unity $\chi_{n,-}^2+\chi_{n,0}^2+\chi_{n,+}^2=1$ such that
$$
\chi_{n,-}=1
\quad\text{on }|z+s_n|\leq R,
\qquad
\operatorname{supp}\chi_{n,-}
\subset
\{|z+s_n|\leq2R\},
$$
$$
\chi_{n,+}(z)=\chi_{n,-}(-z),
$$
and
$$
\sum_{\sigma\in\{-,0,+\}}
|\chi_{n,\sigma}'|^2
\leq
\frac{C}{R^2}.
$$
Then
\begin{equation}\label{twofrontIMS}
\mathcal E_{\lambda_n}(\varphi_n)
=
\sum_{\sigma\in\{-,0,+\}}
\mathcal E_{\lambda_n}(\chi_{n,\sigma}\varphi_n)
-
\int_{\mathbb R}
\left(
\sum_\sigma|\chi_{n,\sigma}'|^2
\right)\varphi_n^2\,dz.
\end{equation}
By \eqref{positiveawayfronts},
$$
\mathcal E_{\lambda_n}(\chi_{n,0}\varphi_n)
\geq
c\|\chi_{n,0}\varphi_n\|_{H^1}^2.
$$
On the left front, set
$$
\phi_{n,-}(y)
:=
(\chi_{n,-}\varphi_n)(y-s_n).
$$
For fixed $R$, define
$$
\varepsilon_n(R)
:=
\sup_{|y|\leq2R}
\left|
W_{\lambda_n}(y-s_n)-W_*(y)
\right|.
$$
Then $\varepsilon_n(R)\longrightarrow0$ as $n\to\infty$ and
$$
\left|
\mathcal E_{\lambda_n}(\chi_{n,-}\varphi_n)
-
\int_{\mathbb R}
\left(
|\phi_{n,-}'|^2+W_*\phi_{n,-}^2
\right)dy
\right|
\leq
\varepsilon_n(R)\|\phi_{n,-}\|_{L^2}^2.
$$
Write
$$
\phi_{n,-}=a_{n,-}q+g_{n,-},
\qquad
\int_{\mathbb R}g_{n,-}q\,dy=0.
$$
Since $\mathscr L_*q=0$,
$$
\int_{\mathbb R}
\left(
|\phi_{n,-}'|^2+W_*\phi_{n,-}^2
\right)dy
=
\int_{\mathbb R}
\left(
|g_{n,-}'|^2+W_*g_{n,-}^2
\right)dy.
$$
Lemma~\ref{singlefrontlem} therefore gives
$$
\mathcal E_{\lambda_n}(\chi_{n,-}\varphi_n)
\geq
c\|g_{n,-}\|_{H^1}^2
-
\varepsilon_n(R)\|\phi_{n,-}\|_{L^2}^2.
$$
Since $(\varphi_n)$ is bounded in $H^1$, the final term is $o(1)$ as $n\to\infty$
for fixed $R$. The same estimate holds on the right front. By symmetry $a_{n,+}=a_{n,-}$. The orthogonality to $\widehat{\tau}_{\lambda_n}$ controls this coefficient.
Indeed, evenness gives
$$
0
=
\int_{\mathbb R}\varphi_n\widehat{\tau}_{\lambda_n}\,dz
=
2\int_{\mathbb R}\varphi_n(z)q(z+s_n)\,dz.
$$
Since $\chi_{n,-}=1$ on $|z+s_n|\leq R$,
$$
\left|
\int_{\mathbb R}
(1-\chi_{n,-})\varphi_n(z)q(z+s_n)\,dz
\right|
\leq
\left(
\int_{|y|\geq R}q(y)^2\,dy
\right)^{1/2}.
$$
Now
$$
\begin{aligned}
0
&=
\int_{\mathbb R}\phi_{n,-}(y)q(y)\,dy
+
\int_{\mathbb R}
(1-\chi_{n,-}(z))
\varphi_n(z)q(z+s_n)\,dz\\
&=
a_{n,-}\|q\|_{L^2(\mathbb R)}^2
+
\int_{\mathbb R}
(1-\chi_{n,-}(z))
\varphi_n(z)q(z+s_n)\,dz,
\end{aligned}
$$
since $g_{n,-}\perp q$. Therefore
$$
|a_{n,-}|
\leq
\frac{1}{\|q\|_{L^2(\mathbb R)}^2}
\left|
\int_{\mathbb R}
(1-\chi_{n,-}(z))
\varphi_n(z)q(z+s_n)\,dz
\right|
\leq
C
\left(
\int_{|y|\geq R}q(y)^2\,dy
\right)^{1/2}.
$$
Combining these estimates with \eqref{twofrontIMS} and using that $\limsup_{n\to\infty}\mathcal E_{\lambda_n}(\varphi_n)\leq0$ gives
$$
\begin{aligned}
\limsup_{n\to\infty}
\Bigl(
&\|g_{n,-}\|_{H^1}^2
+
\|g_{n,+}\|_{H^1}^2
+
\|\chi_{n,0}\varphi_n\|_{H^1}^2
\Bigr)
\leq
\frac{C}{R^2}.
\end{aligned}
$$
Moreover,
$$
\|\varphi_n\|_{H^1}^2
\leq
C\sum_{\sigma\in\{-,0,+\}}
\|\chi_{n,\sigma}\varphi_n\|_{H^1}^2
+
\frac{C}{R^2}\|\varphi_n\|_{L^2}^2,
$$
and
$$
\chi_{n,\pm}\varphi_n
=
a_{n,\pm}q+g_{n,\pm}\Rightarrow\|\chi_{n,\pm}\varphi_n\|_{H^1}
\leq
|a_{n,\pm}|\|q\|_{H^1}
+
\|g_{n,\pm}\|_{H^1}.
$$
Using the estimates for $a_{n,\pm}$, $g_{n,\pm}$, and
$\chi_{n,0}\varphi_n$, together with the boundedness of
$(\varphi_n)$ in $L^2(\mathbb R)$, we obtain
$$
\limsup_{n\to\infty}\|\varphi_n\|_{H^1}^2
\leq
C
\int_{|y|\geq R}q(y)^2\,dy
+
\frac{C}{R^2}.
$$
Letting $R\to\infty$ contradicts
$\|\varphi_n\|_{H^1}=1$.
\end{proof}
\noindent
We now extend this to orthogonality to $\eta_\lambda$.

\begin{lemma}\label{etacoercivity}
There exists $c>0$ such that, for every sufficiently small
$\lambda>0$ and every symmetric $v\in\mathcal H$ satisfying $\langle v,\eta_\lambda\rangle_w=0$, one has
\begin{equation}\label{Hcoereq}
\mathcal E_\lambda(v)
\geq
c\|v\|_{\mathcal H}^2.
\end{equation}
Moreover, if $v\in\mathcal D(\mathcal A_\lambda)\cap\mathcal H$ is symmetric,
then
\begin{equation}\label{reversecoer}
\|\mathcal A_\lambda v\|_{L_w^2}^2
\geq
c\,\mathcal E_\lambda(v).
\end{equation}
\end{lemma}
\begin{proof}
By \eqref{etatwofrontrelation}, the orthogonality condition is equivalent
to
\begin{equation}\label{zorth}
\int_{\mathbb R}
V(z)\zeta_\lambda(z)\widehat{\tau}_\lambda(z)\sech^2z\,dz
=
0.
\end{equation}
Decompose
$$
V=C_1\,\widehat{\tau}_\lambda+V^\perp,
\qquad
\int_{\mathbb R}V^\perp(z)\widehat{\tau}_\lambda(z)\,dz=0.
$$
Lemma~\ref{twofrontunweightedlem} gives
\begin{equation}\label{Vperpcoer}
\mathcal E_\lambda(V^\perp)
\geq
c\|V^\perp\|_{H^1}^2.
\end{equation}
From \eqref{zorth},
$$
C_1
\int_{\mathbb R}
\zeta_\lambda \widehat{\tau}_\lambda^2\sech^2z\,dz
=
-
\int_{\mathbb R}
V^\perp\zeta_\lambda \widehat{\tau}_\lambda\sech^2z\,dz.
$$
Using Sobolev embedding and
Lemma~\ref{twofrontlem},
$$
\left|
\int_{\mathbb R}
V^\perp\zeta_\lambda \widehat{\tau}_\lambda\sech^2z\,dz
\right|
\leq
\|V^\perp\|_{L^\infty}
\int_{\mathbb R}
\zeta_\lambda|\widehat{\tau}_\lambda|\sech^2z\,dz
\leq
C\lambda^2\|V^\perp\|_{H^1}.
$$
Since
$$
\int_{\mathbb R}
\zeta_\lambda \widehat{\tau}_\lambda^2\sech^2z\,dz
\geq
c\lambda^2,
$$
we obtain
\begin{equation}\label{ccontrolled}
|C_1|
\leq
C\|V^\perp\|_{H^1}.
\end{equation}
Now
$$
\mathcal E_\lambda(V)
=
\mathcal E_\lambda(V^\perp)
+
2C_1
\langle\mathscr L_\lambda \widehat{\tau}_\lambda,V^\perp\rangle_{L^2}
+
C_1^2\mathcal E_\lambda(\widehat{\tau}_\lambda).
$$
By Lemma~\ref{twofrontlem}, $\|\mathscr L_\lambda \widehat{\tau}_\lambda\|_{L^2}=o(1)$ and
$$
|\mathcal E_\lambda(\widehat{\tau}_\lambda)|
\leq
\|\mathscr L_\lambda \widehat{\tau}_\lambda\|_{L^2}
\|\widehat{\tau}_\lambda\|_{L^2}
=
o(1).
$$
Together with \eqref{Vperpcoer} and
\eqref{ccontrolled}, this gives
$$
\mathcal E_\lambda(V)
\geq
\bigl(C_1-o(1)\bigr)\|V^\perp\|_{H^1}^2.
$$
For sufficiently small $\lambda$, we thus have $\mathcal E_\lambda(V)
\geq
\frac{C_1}{2}\|V^\perp\|_{H^1}^2$. Also,
$$
\|V\|_{H^1}
\leq
|C_1|\|\widehat{\tau}_\lambda\|_{H^1}
+
\|V^\perp\|_{H^1}
\leq
C\|V^\perp\|_{H^1}.
$$
Hence $\mathcal E_\lambda(V)
\geq
c\|V\|_{H^1}^2$. Using \eqref{H1identity}, this is
\eqref{Hcoereq}.
Suppose now that $v\in\mathcal D(\mathcal A_\lambda)$. By
Cauchy--Schwarz,
$$
\mathcal E_\lambda(V)
=
\langle\mathcal A_\lambda v,v\rangle_w
\leq
\|\mathcal A_\lambda v\|_{L_w^2}
\|v\|_{L_w^2}.
$$
Since $\sech^2z\leq1$,
$$
\|v\|_{L_w^2}^2
=
\int_{\mathbb R}|V|^2\sech^2z\,dz
\leq
\|V\|_{H^1}^2.
$$
The coercivity just proved therefore gives
$$
\mathcal E_\lambda(V)
\leq
C
\|\mathcal A_\lambda v\|_{L_w^2}
\mathcal E_\lambda(V)^{1/2}.
$$
Thus $\|\mathcal A_\lambda v\|_{L_w^2}^2
\geq
c\,\mathcal E_\lambda(V)$, which is \eqref{reversecoer}.
\end{proof}

\begin{remark}
Since $V(z)=v(\theta(z))$ and $\theta\mapsto z$ is a bijection from
$(0,\pi)$ onto $\mathbb R$,
$$
\|v\|_{L^\infty(0,\pi)}
=
\|V\|_{L^\infty(\mathbb R)}.
$$
Hence the Sobolev embedding
$$
H^1(\mathbb R)\hookrightarrow L^\infty(\mathbb R),
$$
together with Lemma~\ref{etacoercivity}, yields
\begin{equation}
\|v\|_{L^\infty(0,\pi)}^2
\leq
C\|v\|_{\mathcal H}^2
\leq
C\mathcal E_\lambda(v)
\end{equation}
whenever $v$ is symmetric and
$\langle v,\eta_\lambda\rangle_w=0$. 
\end{remark}

\subsection{Decay of the Error and Passing through $1$}
The delicate lemma below is the principal work needed to control the non-linear error arising when determining the evolution of $v(t)$ and $\lambda(t)$.

\begin{lemma}\label{weightedAlambda}
For every sufficiently small $\lambda>0$ and every symmetric $v\in\mathcal D(\mathcal A_\lambda)\cap\mathcal H$ satisfying $\langle v,\eta_\lambda\rangle_w=0$, one has
$$
\int_0^\pi\frac{v(\theta)^2}{\sin^3\theta}\,d\theta
\leq
C\|\mathcal A_\lambda v\|_{L_w^2}^2.
$$
\end{lemma}
\begin{proof}
Since $\sin\theta=\sech z$ and $d\theta=\sech z\,dz$, we have
$$
\int_0^\pi\frac{v^2}{\sin^3\theta}\,d\theta
=
\int_{\mathbb R}V^2\cosh^2z\,dz,
\qquad
\|\mathcal A_\lambda v\|_{L_w^2}^2
=
\int_{\mathbb R}|\mathscr L_\lambda V|^2\cosh^2z\,dz.
$$
The symmetry of $v$ implies that $V$ is even and $V'(0)=0$.
We may therefore work on $z\leq0$. Introduce the left-front
coordinate
$$
y=z+s,
\qquad
G(y)=V(y-s),
\qquad
y\leq s.
$$
Then $G'(s)=0$ and
$$
\mathscr L_\lambda V(z)=\mathscr L_sG(y),
$$
Writing
$$
\varpi_s(y):=\cosh^2(y-s)=\sech^2(y-s)^{-1},
$$
it is enough to prove, uniformly for large $s$,
\begin{equation}\label{halflineweightest}
\int_{-\infty}^sG^2\varpi_s\,dy
\leq
C\int_{-\infty}^s|\mathscr L_sG|^2\varpi_s\,dy.
\end{equation}
For later use, we also record the orthogonality condition in these
coordinates. On $y\leq s$, and for large $s$, the second cutoff
comprising $\zeta_\lambda$ vanishes. Since $V$ and
$\eta_\lambda$ are even, the condition
$\langle v,\eta_\lambda\rangle_w=0$ is equivalent to
\begin{equation}\label{halflineweightorth}
\int_{-\infty}^s
G(y)\zeta(y)
\bigl(q(y)+q(2s-y)\bigr)
\sech^2(y-s)\,dy
=
0.
\end{equation}
We first control the region away from the left front. Choose $R$
large enough, and then $s_0$, so that
$$
W_s(y)\geq3
\qquad
\text{whenever }
y\in(-\infty,s]\setminus[-R,R],
\quad s\geq s_0.
$$
Indeed, away from a fixed neighbourhood of $y=0$, the corresponding profile is uniformly close either to $0$ or to $1$, where the
potential equals $4$.
Let $\chi\in C^\infty(\mathbb R)$ satisfy
$$
\chi=0\quad\text{on }[-R,R],
\qquad
\chi=1\quad\text{outside }[-R-1,R+1],
$$
and let
$$
\Omega_R:=\{R-1\leq|y|\leq R+2\}.
$$
We claim that
\begin{equation}\label{exteriorest}
\int_{-\infty}^s\chi^2G^2\varpi_s\,dy
\leq
C\int_{-\infty}^s|\mathscr L_sG|^2\varpi_s\,dy
+
C\int_{\Omega_R}G^2\varpi_s\,dy,
\end{equation}
where $C$ is independent of $s\geq s_0$.
Since $\varpi_s(y)$ grows exponentially as $y\to-\infty$, neither
$G\in H^1(-\infty,s)$ nor $\mathscr L_sG\in L^2(\varpi_s\,dy)$
implies \emph{a priori} that $G\in L^2(\varpi_s\,dy)$. To deal with this, truncate the weight by
$$
\varpi_{s,N}:=\frac{N\varpi_s}{N+\varpi_s}.
$$
Then $\varpi_{s,N}\uparrow\varpi_s$, and direct differentiation
using $\varpi_s''=4\varpi_s-2$ gives $\varpi_{s,N}''\leq4\varpi_{s,N}$. Set $f=\chi G$. The following integration by parts may be
performed after multiplying $f$ by a cutoff at $y=-L$, and then
passing to the limit $L\to\infty$. The resulting cutoff errors
vanish because $f\in H^1(-\infty,s)$ and
$\varpi_{s,N}$ is bounded. We obtain
\begin{equation}\label{nonlin1}
    \int_{-\infty}^s(\mathscr L_sf)f\varpi_{s,N}\,dy=\int_{-\infty}^s|f'|^2\varpi_{s,N}\,dy+
    \int_{-\infty}^s\left(W_s\varpi_{s,N}-\frac12\varpi_{s,N}''\right)f^2\,dy.
\end{equation}
There is no boundary contribution at $y=s$, since
$\chi=1$ near $s$, $G'(s)=0$, and
$\varpi_{s,N}'(s)=0$. On the support of $f$,
$$
W_s\varpi_{s,N}
-\frac12\varpi_{s,N}''
\geq
\varpi_{s,N}.
$$
Throwing away the middle term in \eqref{nonlin1} and applying Cauchy--Schwarz on the left therefore gives $\|f\|_{L^2(\varpi_{s,N}dy)}
\leq
\|\mathscr L_sf\|_{L^2(\varpi_{s,N}dy)}$. Since $\mathscr L_s(\chi G)
=
\chi \mathscr L_sG-2\chi'G'-\chi''G$, it follows that
$$
\int_{-\infty}^s\chi^2G^2\varpi_{s,N}\,dy
\leq
C\int_{-\infty}^s|\mathscr L_sG|^2\varpi_s\,dy
+
C\int_{\operatorname{supp}\chi'}
\bigl(|G'|^2+G^2\bigr)\varpi_s\,dy.
$$
A weighted Caccioppoli estimate controls the derivative term. Choose
$\xi\in C_c^\infty(\Omega_R)$ with
$\xi\equiv1$ on $\operatorname{supp}\chi'$. Multiplying
$$
-G''+W_sG=\mathscr L_sG
$$
by $\xi^2G\varpi_s$ and integrating by parts gives
$$
\begin{aligned}
\int \xi^2|G'|^2\varpi_s
&=
\int (\mathscr L_sG)\xi^2G\varpi_s
-\int W_s\xi^2G^2\varpi_s  \\
&\quad
-2\int \xi\xi'GG'\varpi_s
-\int \xi^2GG'\varpi_s'.
\end{aligned}
$$
On the fixed annulus $\Omega_R$, the quantities
$\|W_s\|_{L^\infty}$, $\|\xi'\|_{L^\infty}$, and
$|\varpi_s'|/\varpi_s$ are bounded uniformly in $s$. Young's
inequality therefore allows the terms containing $G'$ to be
absorbed into the left-hand side, yielding
$$
\int_{\operatorname{supp}\chi'}|G'|^2\varpi_s\,dy
\leq
C\int_{\Omega_R}
\bigl(|\mathscr L_sG|^2+G^2\bigr)\varpi_s\,dy.
$$
The constant is uniform in $s$: on $\Omega_R$, the potentials $W_s$
are uniformly bounded and $\left|\frac{\varpi_s'}{\varpi_s}\right|\leq2$. Letting $N\to\infty$ proves
\eqref{exteriorest}. 

Note that the left-hand side of
\eqref{halflineweightest} is finite whenever its
right-hand side is finite. Indeed, the weighted norm of $G$ on
$[-R-1,R+1]$ is finite because $G\in H^1_{\mathrm{loc}}$ and
$\varpi_s$ is bounded there. We prove \eqref{halflineweightest} by contradiction.
Suppose it fails. Then there are $s_n\to\infty$, with
$\lambda_n=e^{-s_n}$, and functions $G_n$ satisfying
$G_n'(s_n)=0$ and
\eqref{halflineweightorth}, such that
\begin{equation}\label{weightednorm}
\int_{-\infty}^{s_n}G_n^2\varpi_{s_n}\,dy=1,
\qquad
\int_{-\infty}^{s_n}|\mathscr L_{s_n}G_n|^2
\varpi_{s_n}\,dy\longrightarrow0.
\end{equation}
Set
$$
\widetilde G_n:=\lambda_n^{-1}G_n.
$$
For every compact interval $K\subset\mathbb R$, we have $\lambda_n^2\varpi_{s_n}(y)
\longrightarrow
\frac14e^{-2y}$ uniformly on $K$. The normalisation therefore gives a uniform
$L^2(K)$-bound for $\widetilde G_n$. Moreover,
the second part of \eqref{weightednorm} gives
$$
\|\mathscr L_{s_n}\widetilde G_n\|_{L^2(K)}\longrightarrow0.
$$
Since $W_{s_n}\to W_*=4\bigl(6Q(y)^2-6Q(y)+1\bigr)$ smoothly on compact sets, interior elliptic estimates give, after passing to a subsequence,
$$
\widetilde G_n\longrightarrow G_\infty
\qquad
\text{in }H^1_{\mathrm{loc}}(\mathbb R),
$$
and
$$
\mathscr L_*G_\infty=0,
\qquad
\mathscr L_*=-\partial_y^2+W_*.
$$
Fatou's lemma applied to $G_n^2\varpi_{s_n}
=
\widetilde G_n^2\lambda_n^2\varpi_{s_n}$ shows that
$$
\int_{\mathbb R}G_\infty(y)^2e^{-2y}\,dy\leq4.
$$
One solution of $\mathscr L_*G_\infty=0$ is $q$. A second, linearly independent
solution is
$$
q(y)\int_0^y\frac{d\sigma}{q(\sigma)^2},
$$
which does not belong to $L^2(e^{-2y}\,dy)$. Hence $G_\infty=dq$ for some $d\in\mathbb R$.
The orthogonality condition determines $d$. Substituting
$G_n=\lambda_n\widetilde G_n$ into
\eqref{halflineweightorth} and multiplying by
$\lambda_n^{-3}$ gives
$$
0
=
\int_{-\infty}^{s_n}
\widetilde G_n(y)\zeta(y)
\bigl(q(y)+q(2s_n-y)\bigr)
\lambda_n^{-2}\sech^2(y-s_n)\,dy.
$$
On the fixed support of $\zeta$,
$$
q(2s_n-y)\longrightarrow0,
\qquad
\lambda_n^{-2}\sech^2(y-s_n)
\longrightarrow4e^{2y}
$$
uniformly. Passing to the limit yields
$$
0
=
4d\int_{\mathbb R}\zeta(y)q(y)^2e^{2y}\,dy.
$$
The integral is strictly positive, so $d=0$. Therefore
\begin{equation}\label{Gntozero}
\widetilde G_n\longrightarrow0
\qquad
\text{in }H^1_{\mathrm{loc}}(\mathbb R).
\end{equation}
It remains to rule out escape of the normalised mass away from the
front. Apply \eqref{exteriorest} to $G_n$. The operator
term tends to zero by \eqref{weightednorm}. On the fixed
annulus $\Omega_R$, $G_n^2\varpi_{s_n}
=
\widetilde G_n^2\lambda_n^2\varpi_{s_n}$, and the second factor converges uniformly. Hence
\eqref{Gntozero} gives
$$
\int_{\Omega_R}G_n^2\varpi_{s_n}\,dy\longrightarrow0.
$$
It follows from \eqref{exteriorest} that
$$
\int_{-\infty}^{s_n}\chi^2G_n^2\varpi_{s_n}\,dy
\longrightarrow0.
$$
The same local convergence gives
$$
\int_{|y|\leq R+1}G_n^2\varpi_{s_n}\,dy
\longrightarrow0.
$$
These two regions cover $(-\infty,s_n]$, up to overlap, and so
$$
\int_{-\infty}^{s_n}G_n^2\varpi_{s_n}\,dy\longrightarrow0,
$$
contradicting the normalisation in
\eqref{weightednorm}. This proves
\eqref{halflineweightest}, and hence the lemma.
\end{proof}

From this point onward we work on the maximal tubular interval $[0,T)$ provided by Lemma~\ref{tubeee_eta}. Thus
$$
a(\theta,t)
=
u_{\lambda(t)}(\theta)+v(\theta,t),
\qquad
\langle v(t),\eta_{\lambda(t)}\rangle_w=0.
$$
To study the evolution of the error, we evaluate the quadratic energy along the flow:
$$
E(t)
:=
\mathcal E_{\lambda(t)}(v(t))
=
\langle\mathcal A_{\lambda(t)}v(t),v(t)\rangle_w,
$$
together with the dissipation
$$
X(t)
:=
\|\mathcal A_{\lambda(t)}v(t)\|_{L_w^2}.
$$
Differentiating $a=u_\lambda+v$ gives
$$
v_t
=
L(u_\lambda+v)-\lambda'\tau_\lambda.
$$
Set
$$
\mathcal N_\lambda(v)
:=
\frac{
f(u_\lambda+v)-f(u_\lambda)-f'(u_\lambda)v
}{\sin^2\theta},
$$
where $f$ was defined in \eqref{nonlinpolyeq}. Then
\begin{equation}\label{erroreqation}
v_t
=
-\mathcal A_\lambda v
+
L(u_\lambda)
+
\mathcal N_\lambda(v)
-
\lambda'\tau_\lambda.
\end{equation}

\begin{corollary}\label{OGquaderror}
For sufficiently small $\lambda$, whenever $v\in\mathcal D(\mathcal A_\lambda)\cap\mathcal H$ is symmetric and $\langle v,\eta_\lambda\rangle_w=0$, one has
$$
\|\mathcal N_\lambda(v)\|_{L_w^2}
\leq
C\|\mathcal A_\lambda v\|_{L_w^2}
\bigl(E^{1/2}+E\bigr).
$$
\end{corollary}
\begin{proof}
A direct expansion gives
$$
\mathcal N_\lambda(v)
=
\frac{12(1-2u_\lambda)v^2-8v^3}{\sin^2\theta}.
$$
Hence
$$
\|\mathcal N_\lambda(v)\|_{L_w^2}^2
\leq
C
\bigl(
\|v\|_{L^\infty}^2+\|v\|_{L^\infty}^4
\bigr)
\int_0^\pi\frac{v^2}{\sin^3\theta}\,d\theta.
$$
By Lemma~\ref{etacoercivity} and Sobolev embedding $\|v\|_{L^\infty}^2\leq CE$. Applying Lemma~\ref{weightedAlambda} gives
$$
\|\mathcal N_\lambda(v)\|_{L_w^2}^2
\leq
C(E+E^2)
\|\mathcal A_\lambda v\|_{L_w^2}^2.
$$
Taking square roots and using an elementary inequality proves the result.
\end{proof}

\begin{theorem}\label{errordecaythrm}
Assume
$$
E(t)\leq\varepsilon
$$
for $0\leq t<T$, where $\varepsilon>0$ is sufficiently small. Then there exists $c>0$, independent of $t$ and of sufficiently small $\lambda$, such that
\begin{equation}\label{energydiss}
E'(t)\leq-cX(t)^2+C\lambda(t)^6.
\end{equation}
\end{theorem}
\begin{proof}
The calculations below are first justified on any compact interval
$[\tau,T_0]\subset(0,T)$. On such an interval, parabolic regularity gives the
required regularity of $v$, while Lemma~\ref{tubeee_eta} gives
$\lambda\in C^1$. We differentiate $E(t)$ through the quadratic form associated
with $A_{\lambda(t)}$, whose form domain $H$ is independent of $\lambda$; thus
no differentiation of the domain of the unbounded operator is involved.
Integrating the resulting identities on $[\tau,t]$ and letting $\tau\downarrow0$,
the stated inequalities extend to $t=0$ by the continuity of $v$ in $H$ and of
$\lambda$.
We suppress the time dependence and write $\lambda=\lambda(t),
\,
v=v(\cdot,t),
\,
X=\|\mathcal A_\lambda v(t)\|_{L_w^2}$.

Differentiating the orthogonality condition and using
\eqref{erroreqation} gives
$$
\lambda'
\left(
\langle\tau_\lambda,\eta_\lambda\rangle_w-\langle v,\partial_\lambda\eta_\lambda\rangle_w
\right)
=
-\langle\mathcal A_\lambda v,\eta_\lambda\rangle_w
+
\langle L(u_\lambda),\eta_\lambda\rangle_w
+
\langle\mathcal N_\lambda(v),\eta_\lambda\rangle_w.
$$
By Proposition \ref{norms}, $\langle\tau_\lambda,\eta_\lambda\rangle_w\geq c_1\lambda^{-2}$
for some $c_1>0$, while
$$
\left|
\langle v,\partial_\lambda\eta_\lambda\rangle_w
\right|
\leq
\|v\|_{L^\infty}
\|\partial_\lambda\eta_\lambda\|_{L_w^1}
\leq
C\lambda^\delta\lambda^{-2},
$$
using Lemma \ref{tubeee_eta} and Proposition \ref{norms} again.
After decreasing the tubular scale $\lambda_*$, the coefficient of
$\lambda'$ is therefore bounded below by
$\frac12c_1\lambda^{-2}$.
From Proposition \ref{norms} one immediately sees that
$$
\|\eta_\lambda\|_{L_w^2}\leq C\lambda^{-2},
\qquad
\|L(u_\lambda)\|_{L_w^2}\leq C\lambda^3.
$$
Using this alongside Corollary~\ref{OGquaderror}, we obtain
\begin{equation}\label{lambdaprimebound}
|\lambda'|
\leq
C\left(
\lambda^3+X+X(E^{1/2}+E)
\right).
\end{equation}
We now differentiate $E$. Since
$$
v_t
=
-\mathcal A_\lambda v
+
L(u_\lambda)
+
\mathcal N_\lambda(v)
-
\lambda'\tau_\lambda,
$$
self-adjointness gives
\begin{align}
\frac12E'
&=
-X^2
+
\langle\mathcal A_\lambda v,L(u_\lambda)\rangle_w
+
\langle\mathcal A_\lambda v,\mathcal N_\lambda(v)\rangle_w
\notag\\
&\quad
-\lambda'\langle\mathcal A_\lambda v,\tau_\lambda\rangle_w
+
\frac12\lambda'
\langle\partial_\lambda\mathcal A_\lambda v,v\rangle_w.
\label{Eeqn_eta}
\end{align}
The estimate on $\|L(u_\lambda)\|_{L^2_w}$ in Proposition \ref{norms} gives
$$
\left|
\langle\mathcal A_\lambda v,L(u_\lambda)\rangle_w
\right|
\leq
C\lambda^3X
\leq
\frac18X^2+C\lambda^6.
$$
Similarly, Corollary~\ref{OGquaderror} gives
$$
\left|
\langle\mathcal A_\lambda v,\mathcal N_\lambda(v)\rangle_w
\right|
\leq
CX^2(E^{1/2}+E)
\leq
\frac18X^2
$$
once $\varepsilon$ is sufficiently small.
For the first modulation term, self-adjointness,
Proposition~\ref{norms}, and coercivity give
$$
\left|
\langle\mathcal A_\lambda v,\tau_\lambda\rangle_w
\right|
=
\left|
\langle v,\mathcal A_\lambda\tau_\lambda\rangle_w
\right|
\leq
C\lambda^2E^{1/2}.
$$
Combining this with \eqref{lambdaprimebound},
\begin{equation*}
\begin{aligned}
\left|
\lambda'\langle\mathcal A_\lambda v,\tau_\lambda\rangle_w
\right|
&\leq
C\lambda^2E^{1/2}
\left(
\lambda^3+X+X(E^{1/2}+E)
\right)\\
&\leq
\frac1{16}X^2+C(\lambda^6+E^2).
\end{aligned}
\end{equation*}
Here the terms containing $X(E^{1/2}+E)$ are absorbed using
$E\leq\varepsilon$, while Young's inequality gives
$\lambda^2E^{1/2}X\leq \frac1{32}X^2+C\lambda^4E$ and
$\lambda^4E\leq C(\lambda^6+E^2)$.
It remains to estimate the variation of the operator. Since
$$
\partial_\lambda\mathcal A_\lambda
=
\frac{24(2u_\lambda-1)\tau_\lambda}{\sin^2\theta},
$$
the pointwise estimate $|\tau_\lambda(\theta)|\sin^2\theta\leq C\lambda$ from Proposition \ref{norms} and Lemma~\ref{weightedAlambda} imply
$$
\left|
\langle\partial_\lambda\mathcal A_\lambda v,v\rangle_w
\right|
\leq
C\lambda
\int_0^\pi\frac{v^2}{\sin^3\theta}\,d\theta
\leq
C\lambda X^2.
$$
On the other hand, $\|\tau_\lambda\|_{L^\infty}\leq C\lambda^{-1}$
and Lemma~\ref{etacoercivity} give
$$
\left|
\langle\partial_\lambda\mathcal A_\lambda v,v\rangle_w
\right|
\leq
C\lambda^{-1}
\int_0^\pi\frac{v^2}{\sin\theta}\,d\theta
\leq
C\lambda^{-1}E.
$$
Taking the geometric mean of these two bounds yields
$$
\left|
\langle\partial_\lambda\mathcal A_\lambda v,v\rangle_w
\right|
\leq
CE^{1/2}X.
$$
Using \eqref{lambdaprimebound}, we therefore obtain
\begin{equation*}
\begin{aligned}
\left|
\frac12\lambda'
\langle\partial_\lambda\mathcal A_\lambda v,v\rangle_w
\right|
&\leq
CE^{1/2}X
\left(
\lambda^3+X+X(E^{1/2}+E)
\right)\\
&\leq
\frac1{16}X^2+C\lambda^6,
\end{aligned}
\end{equation*}
after decreasing $\varepsilon$. Indeed, the last two terms are
$O(\varepsilon^{1/2}X^2)$, while Young's inequality controls
$\lambda^3E^{1/2}X$.
Substituting these estimates into \eqref{Eeqn_eta} gives
$$
\frac12E'
\leq
-\frac{5}{8}X^2+C(\lambda^6+E^2).
$$
By the reverse coercivity estimate in
Lemma~\ref{etacoercivity}, $E\leq CX^2$. Hence $E^2\leq\varepsilon E\leq C\varepsilon X^2$ and this term may be absorbed after decreasing $\varepsilon$.
Thus
$$
E'\leq-\frac{5}{8}X^2+C\lambda^6,
$$
which proves \eqref{energydiss}.
\end{proof}
\begin{remark}\label{energydecayremark}
By the reverse coercivity estimate in Lemma~\ref{etacoercivity} $X(t)^2\geq c\,E(t)$. Hence Theorem~\ref{errordecaythrm} implies
$$
E'(t)\leq-c_{\mathrm{decay}} E(t)+C\lambda(t)^6
$$
for some $c_{\mathrm{decay}}>0$. Multiplying by $e^{c_{\mathrm{decay}} t}$ and integrating gives
\begin{equation}\label{Edecay}
E(t)
\leq
e^{-c_{\mathrm{decay}} t}E(0)
+
C\sup_{0\leq \tau\leq t}\lambda(\tau)^6.
\end{equation}
\end{remark}

We next control the change in the scale parameter.

\begin{lemma}
\label{logscaledrift}
Fix $K>0$, and set $\lambda_0:=\lambda(0)$. Suppose that
$$
T\leq K\left|\log\lambda_0\right|,
$$
and that
$$
\frac12\lambda_0\leq\lambda(t)\leq2\lambda_0,
\qquad
E(t)\leq\varepsilon
$$
for every $t\in[0,T]$, where $\varepsilon>0$ is sufficiently
small. Assume also that $E(0)=o(1)$ as $\lambda_0\downarrow0$, then
$$
\lambda(t)=(1+o(1))\lambda_0
$$
uniformly on $[0,T]$.
\end{lemma}
\begin{proof}
All $o(1)$-terms below are as $\lambda_0\downarrow0$, uniformly
for $t\in[0,T]$; their constants may depend on $K$, but not on
$\lambda_0$. We test the error equation against a normalized version of the
tangent vector $\tau_\lambda=\partial_\lambda u_\lambda$. Let
$$
g_\lambda:=-\frac{3}{16\lambda}\tau_\lambda.
$$
The normalization is chosen so that, by Proposition~\ref{norms},
\begin{equation}\label{logpairing}
-\langle\lambda\tau_\lambda,g_\lambda\rangle_w
=
\frac{3\|\tau_\lambda\|_{L_w^2}^2}{16}
=
1+O\left(\lambda^2\log\frac1\lambda\right).
\end{equation}
The same proposition gives
\begin{equation}\label{gests}
\|g_\lambda\|_{L_w^2}\leq C\lambda^{-1},
\qquad
\|\mathcal A_\lambda g_\lambda\|_{L_w^2}\leq C\lambda,
\qquad
|g_\lambda(\theta)|\sin^2\theta\leq C.
\end{equation}
We also record $\|g_\lambda\|_{L_w^1}\leq C\left|\log\lambda\right|$ which can be proved by direct computation. The last and most crucial estimate we need is
\begin{equation}\label{partialsbound}
\left(
\int_0^\pi
|\partial_sg_\lambda|^2\sin^5\theta\,d\theta
\right)^{1/2}
\leq C\lambda.
\end{equation}
To verify it, recall from \eqref{twofrontrelation} that $\tau_\lambda(\theta(z))
=
-\lambda^{-1}\widehat{\tau}_\lambda(z)$. Thus
$$
g_\lambda(\theta(z))=\frac{3\lambda^{-2}}{16}\widehat{\tau}_\lambda(z)
$$
and
$$
\partial_sg_\lambda(\theta(z))
=
\frac{3\lambda^{-2}}{16}
\bigl(2\widehat{\tau}_\lambda(z)+\partial_s\widehat{\tau}_\lambda(z)\bigr).
$$
The function $2q+q'$ satisfies
$$
|2q(y)+q'(y)|
\leq
\begin{cases}
Ce^{2y},&y\leq0,\\
Ce^{-4y},&y\geq0.
\end{cases}
$$
Using $\sin^5\theta\,d\theta=\sech^6z\,dz$, symmetry, and the
change of variables $y=z+s$, these bounds give
{\small
\begin{align*}
\int_{\mathbb R}
|2\widehat{\tau}_\lambda+\partial_s\widehat{\tau}_\lambda|^2\sech^6z\,dz
&\leq
C\int_{-\infty}^s
\left(
|2q(y)+q'(y)|^2
+
|2q(2s-y)+q'(2s-y)|^2
\right)
\sech^6(y-s)\,dy\\
&\leq
Ce^{-6s}.    
\end{align*}
}
Since $\lambda=e^{-s}$, multiplication by the factor
$\lambda^{-4}$ coming from $\partial_sg_\lambda$ proves
\eqref{partialsbound}.
We first collect the time-integrated estimates supplied by
Theorem~\ref{errordecaythrm}. Integrating
\eqref{energydiss}, using $E(T)\geq0$ and
$\lambda(t)\leq2\lambda_0$, gives
\begin{equation}\label{integrateddissdrift}
\int_0^T X(t)^2\,dt
\leq
C E(0)+C\lambda_0^6T
=o(1).
\end{equation}
Similarly \eqref{Edecay} gives
\begin{equation}\label{Eintegral}
\int_0^T E(t)^{1/2}\,dt
\leq
CE(0)^{1/2}+C\lambda_0^3T
=o(1).
\end{equation}
Since $s'=-\lambda'/\lambda$, the error equation
\eqref{erroreqation} may be written as
$$
v_t
=
-\mathcal A_\lambda v
+
L(u_\lambda)
+
\mathcal N_\lambda(v)
+
s'\lambda\tau_\lambda.
$$
Differentiating $\langle v(t),g_{\lambda(t)}\rangle_w$, and observing that $g_\lambda$ is parametrized by $s$, gives
$$
\begin{aligned}
\frac{d}{dt}\langle v,g_\lambda\rangle_w
={}&
-\langle\mathcal A_\lambda v,g_\lambda\rangle_w
+
\langle L(u_\lambda)+\mathcal N_\lambda(v),g_\lambda\rangle_w\\
&+
s'\langle\lambda\tau_\lambda,g_\lambda\rangle_w
+
s'\langle v,\partial_sg_\lambda\rangle_w.
\end{aligned}
$$
Using \eqref{logpairing} and rearranging, we obtain the
exact identity
\begin{align}
s'
={}&
-\langle\mathcal A_\lambda v,g_\lambda\rangle_w
+
\langle L(u_\lambda)+\mathcal N_\lambda(v),g_\lambda\rangle_w
-\frac{d}{dt}\langle v,g_\lambda\rangle_w
\notag\\
&+
s'\langle v,\partial_sg_\lambda\rangle_w
+
s'
\left(
1-\frac{3\|\tau_\lambda\|_{L_w^2}^2}{16}
\right).
\label{sprime}
\end{align}
We estimate the terms on the right after integration in time.
First, self-adjointness, \eqref{gests}, and coercivity
give
$$
\left|
\langle\mathcal A_\lambda v,g_\lambda\rangle_w
\right|
=
\left|
\langle v,\mathcal A_\lambda g_\lambda\rangle_w
\right|
\leq
C\lambda E^{1/2}.
$$
Hence, by \eqref{Eintegral},
$$
\int_0^T
\left|
\langle\mathcal A_\lambda v,g_\lambda\rangle_w
\right|dt
=o(1).
$$
Similarly,
$$
\left|
\langle L(u_\lambda),g_\lambda\rangle_w
\right|
\leq
\|L(u_\lambda)\|_{L_w^2}
\|g_\lambda\|_{L_w^2}
\leq
C\lambda^2,
$$
and therefore its integral over $[0,T]$ is $o(1)$.
For the nonlinear term, the pointwise bound
$|g_\lambda|\sin^2\theta\leq C$ gives
$$
\begin{aligned}
\left|
\langle\mathcal N_\lambda(v),g_\lambda\rangle_w
\right|
&\leq
C\int_0^\pi
\frac{v^2+|v|^3}{\sin^3\theta}\,d\theta\\
&\leq
C(1+\|v\|_{L^\infty})X^2
\leq
C(1+E^{1/2})X^2.
\end{aligned}
$$
Here we used Lemma~\ref{weightedAlambda} and the Sobolev
estimate $\|v\|_{L^\infty}\leq CE^{1/2}$. It follows from
\eqref{integrateddissdrift} that
$$
\int_0^T
\left|
\langle\mathcal N_\lambda(v),g_\lambda\rangle_w
\right|dt
=o(1).
$$
The weighted estimate and \eqref{partialsbound} give
$$
\left|
\langle v,\partial_sg_\lambda\rangle_w
\right|
\leq
\left(
\int_0^\pi\frac{v^2}{\sin^3\theta}\,d\theta
\right)^{1/2}
\left(
\int_0^\pi
|\partial_sg_\lambda|^2\sin^5\theta\,d\theta
\right)^{1/2}
\leq
C\lambda X.
$$
Dividing \eqref{lambdaprimebound} by $\lambda$ yields
$$
|s'|
\leq
C\left(
\lambda^2
+
\lambda^{-1}X
+
\lambda^{-1}X(E^{1/2}+E)
\right).
$$
Since $E\leq\varepsilon$,
$$
|s'|
\left|
\langle v,\partial_sg_\lambda\rangle_w
\right|
\leq
C\lambda^3X+CX^2.
$$
The $X^2$-term is integrable by
\eqref{integrateddissdrift}, while
$$
\int_0^T\lambda^3X\,dt
\leq
C\lambda_0^3T^{1/2}
\left(\int_0^TX^2\,dt\right)^{1/2}
=o(1).
$$
Thus
$$
\int_0^T
|s'|
\left|
\langle v,\partial_sg_\lambda\rangle_w
\right|dt
=o(1).
$$
It remains to control the small normalization error. By
Proposition~\ref{norms},
$$
\left|
1-\frac{3\|\tau_\lambda\|_{L_w^2}^2}{16}
\right|
\leq
C\lambda^2\log\frac1\lambda.
$$
Using the preceding estimate for $s'$ and
$\lambda(t)\asymp\lambda_0$, we obtain
$$
\begin{aligned}
\int_0^T
|s'|
\left|
1-\frac{3\|\tau_\lambda\|_{L_w^2}^2}{16}
\right|dt
&\leq
C\lambda_0^4\log\frac1{\lambda_0}\,T\\
&\quad+
C\lambda_0\log\frac1{\lambda_0}\int_0^TX\,dt\\
&\quad+
C\lambda_0\log\frac1{\lambda_0}
\int_0^TX(E^{1/2}+E)\,dt.
\end{aligned}
$$
The first term tends to zero because
$T\leq K|\log\lambda_0|$. For the second,
Cauchy--Schwarz gives
$$
\lambda_0\log\frac1{\lambda_0}\int_0^TX\,dt
\leq
\lambda_0\log\frac1{\lambda_0}\,
T^{1/2}
\left(\int_0^TX^2\,dt\right)^{1/2}
=o(1).
$$
The last term is bounded in the same way, since
$E^{1/2}+E\leq C\varepsilon^{1/2}$. Hence the normalization-error
term in \eqref{sprime} also has integral $o(1)$.
Finally, Lemma~\ref{tubeee_eta} and the bound recorded after
\eqref{gests} imply
$$
|\langle v(t),g_{\lambda(t)}\rangle_w|
\leq
\|v(t)\|_{L^\infty}
\|g_{\lambda(t)}\|_{L_w^1}
\leq
C\lambda(t)^\delta
\left|\log\lambda(t)\right|
=o(1)
$$
uniformly on $[0,T]$.
Integrating \eqref{sprime} from $0$ to any
$t\leq T$, every term on the right is $o(1)$, including the boundary term
$$
\langle v(t),g_{\lambda(t)}\rangle_w
-
\langle v(0),g_{\lambda_0}\rangle_w.
$$
Therefore $\sup_{0\leq t\leq T}|s(t)-s(0)|=o(1)$. Exponentiating gives
$$
\lambda(t)=e^{o(1)}\lambda_0=(1+o(1))\lambda_0
$$
uniformly on $[0,T]$.
\end{proof}

\begin{theorem}\label{eventualdropthm}
Let
$$
a_0(\theta)
=
u_{\lambda_0}(\theta)
+
\lambda_0^\alpha\psi_{\lambda_0}(\theta)
$$
be the initial data constructed in
Lemma~\ref{perturbboundlem} with $\psi_{\lambda_0}$ from Remark \ref{perturbationcutoff}. Let $a(\theta,t)$ denote
the corresponding solution. Then, for all
sufficiently small $\lambda_0>0$, there exists $t_*>0$ such that
$$
\sup_{\theta\in[0,\pi]}a(\theta,t_*)<1.
$$
\end{theorem}
\begin{proof}
Choose
$$
0<\delta<\alpha.
$$
Set
$$
t_*:=\frac{6}{c_{\mathrm{decay}}}\left|\log\lambda_0\right|,
$$
where $c_{\mathrm{decay}}$ is the constant from
Remark~\ref{energydecayremark}.
By Remark~\ref{initialscalerem},
$$
\lambda(0)=\lambda_0,
\qquad
v(\theta,0)
=
\lambda_0^\alpha\psi_{\lambda_0}(\theta).
$$
Moreover, Remark~\ref{perturbationcutoff} gives
$$
\|v(\cdot,0)\|_{\mathcal H}^2
\leq
C\lambda_0^{2\alpha}
\left|\log\lambda_0\right|.
$$
Since $\mathcal E_{\lambda_0}$ is bounded
above by $C\|\cdot\|_{\mathcal H}^2$,
\begin{equation}\label{initialenergy}
E(0)
\leq
C\lambda_0^{2\alpha}
\left|\log\lambda_0\right|
=
o(1).
\end{equation}
In particular, since $\delta<\alpha$, the initial condition belongs
to $\mathcal U_{\lambda_*,\delta}$ for all sufficiently small
$\lambda_0$.
Let $[0,T)$ be the maximal tubular interval from
Lemma~\ref{tubeee_eta}. We decrease $\lambda_0$ so that $2\lambda_0<\frac{\lambda_*}{2}$. Define $\kappa$ to be the supremum of the times
$t<\min\{T,t_*\}$ for which, throughout $[0,t]$,
$$
\frac12\lambda_0<\lambda(\tau)<2\lambda_0,
\qquad
E(\tau)<\varepsilon,
\qquad
\|v(\cdot,\tau)\|_{L^\infty}<\lambda(\tau)^\delta.
$$
Here $\varepsilon$ is sufficiently small for 
Theorem~\ref{errordecaythrm} and Lemma \ref{logscaledrift}. By
\eqref{initialenergy}, $\lambda(0)=\lambda_0$, and
$\delta<\alpha$, these inequalities hold on a non-empty initial
interval.
For $t<\kappa$, \eqref{Edecay} gives
$$
E(t)
\leq
e^{-c_{\mathrm{decay}} t}E(0)
+
C\sup_{0\leq \tau\leq t}\lambda(\tau)^6
\leq
e^{-c_{\mathrm{decay}} t}E(0)+C\lambda_0^6.
$$
Thus $E(t)=o(1)$ uniformly for $t<\kappa$, and hence, after
decreasing $\lambda_0$ we have $E(t)\leq\frac{\varepsilon}{2}$. We next control the scale. For every $t<\kappa$, the hypotheses of
Lemma~\ref{logscaledrift} hold on $[0,t]$, with the
fixed constant $K=6/c_{\mathrm{decay}}$. Therefore $\lambda(t)=(1+o(1))\lambda_0$ uniformly for $t<\kappa$. In particular, for sufficiently small
$\lambda_0$,
$$
\frac34\lambda_0
\leq
\lambda(t)
\leq
\frac32\lambda_0.
$$
Finally, coercivity and Sobolev embedding give $\|v(\cdot,t)\|_{L^\infty}
\leq
CE(t)^{1/2}$. Using the preceding energy estimate and
\eqref{initialenergy},
$$
\|v(\cdot,t)\|_{L^\infty}
\leq
Ce^{-c_{\mathrm{decay}} t/2}
\lambda_0^\alpha
\left|\log\lambda_0\right|^{1/2}
+
C\lambda_0^3.
$$
Since $\delta<\alpha<2<3$, the right-hand side is
$o(\lambda_0^\delta)$, uniformly for $t<\kappa$. As
$\lambda(t)\asymp\lambda_0$, it follows that
$$
\|v(\cdot,t)\|_{L^\infty}
\leq
\frac12\lambda(t)^\delta
$$
for sufficiently small $\lambda_0$.
All three bootstrap inequalities have therefore been improved by a
fixed margin. Hence $\kappa=\min\{T,t_*\}$. We claim that $T>t_*$. Suppose instead that $T\leq t_*$. Choose
$t_n\uparrow T$. The improved scale bounds allow us, after passing
to a subsequence, to assume that
$$
\lambda(t_n)\longrightarrow\bar\lambda
\in
\left[\frac34\lambda_0,\frac32\lambda_0\right]
\subset
\left(0,\frac{\lambda_*}{2}\right].
$$
The solution exists smoothly beyond $T$, and hence
$a(\cdot,t_n)\to a(\cdot,T)$ in $L^\infty$. Since
$$
a(\cdot,t_n)
=
u_{\lambda(t_n)}+v(\cdot,t_n),
\qquad
\|v(\cdot,t_n)\|_{L^\infty}
\leq
\frac12\lambda(t_n)^\delta,
$$
passing to the limit gives
$$
\|a(\cdot,T)-u_{\bar\lambda}\|_{L^\infty}
\leq
\frac12\bar\lambda^\delta
<
\bar\lambda^\delta.
$$
Thus $a(\cdot,T)\in\mathcal U_{\lambda_*,\delta}$. This set is open
in $L^\infty$, so continuity of the flow implies that
$a(\cdot,t)\in\mathcal U_{\lambda_*,\delta}$ for times slightly
larger than $T$, contradicting the maximality of the tubular
interval. Therefore $T>t_*$ and $\kappa=t_*$.

At time $t_*$, the energy estimate gives
$$
\begin{aligned}
E(t_*)
&\leq
e^{-c_{\mathrm{decay}} t_*}E(0)+C\lambda_0^6\\
&=
\lambda_0^6E(0)+C\lambda_0^6
=
O(\lambda_0^6).
\end{aligned}
$$
Consequently,
$$
\|v(\cdot,t_*)\|_{L^\infty}
\leq
CE(t_*)^{1/2}
\leq
C\lambda_0^3.
$$ 
Lemma \ref{logscaledrift} also gives
$$
\lambda(t_*)=(1+o(1))\lambda_0.
$$
The profile $u_\lambda$ is maximized at the equator, where
$$
\max_{\theta\in[0,\pi]}u_\lambda(\theta)
=
u_\lambda\left(\frac{\pi}{2}\right)
=
\frac{1-\lambda^2}{1+\lambda^2}.
$$
It follows that
$$
\max_{\theta\in[0,\pi]}u_{\lambda(t_*)}(\theta)
=
1-2\lambda_0^2+o(\lambda_0^2).
$$
Therefore
$$
\begin{aligned}
\sup_{\theta\in[0,\pi]}a(\theta,t_*)
&\leq
\max_{\theta\in[0,\pi]}u_{\lambda(t_*)}(\theta)
+
\|v(\cdot,t_*)\|_{L^\infty}\\
&=
1-2\lambda_0^2+o(\lambda_0^2)
<
1
\end{aligned}
$$
for all sufficiently small $\lambda_0$.
\end{proof}

\section*{Acknowledgements}

I am grateful to Jason Lotay and Andrew Dancer for many helpful discussions,
and for their careful reading of and valuable comments on the manuscript.

\bigskip

\bibliographystyle{amsplain} 
\bibliography{refs}

\end{document}